\documentclass[11points]{amsart}
\usepackage{amsmath}
\usepackage{amssymb}
\usepackage{amsfonts}
\usepackage{graphicx}
\usepackage{verbatim}
\usepackage[
textwidth=14.5cm, 
textheight=21cm,
hmarginratio=1:1,
vmarginratio=1:1]{geometry}

 \newtheorem{thm}{Theorem}[section]
 
 \newtheorem{lem}[thm]{Lemma}
 \newtheorem{prop}[thm]{Proposition}
 \theoremstyle{definition}

 \theoremstyle{remark}
 \newtheorem{rem}[thm]{Remark}
 
\usepackage[OT2,T1]{fontenc}
\DeclareSymbolFont{cyrletters}{OT2}{wncyr}{m}{n}
\DeclareMathSymbol{\Lfun}{\beta}{cyrletters}{"62}
\DeclareMathSymbol{\berd}{\beta}{cyrletters}{"42}
\numberwithin{equation}{section}
\numberwithin{figure}{section}

\newcommand{\e}{\mathrm e}

\newcommand{\C}{{\mathbb C}}

\newcommand{\esssup}{\,{\rm ess}\,{\rm sup}\,}

\newcommand{\diff}{{\mathrm d}}
\newcommand{\imag}{{\mathrm i}}

\newcommand{\dist}{\operatorname{dist}}

\begin{document}
%
\title[Bulk asymptotics for polyanalytic correlation kernels]
{Bulk asymptotics for polyanalytic correlation kernels}
\author[Haimi]
{Antti Haimi}

\address{Haimi: Department of Mathematics\\
The Royal Institute of Technology\\
S -- 100 44 Stockholm\\
SWEDEN}

\email{anttih@kth.se}

\date{\today}

\begin{abstract}
For a weight function $Q:\mathbb{C} \to \mathbb{R}$ and a positive scaling parameter $m$, we study reproducing kernels $K_{q,mQ,n}$of the polynomial spaces
\[ A^2_{q,mQ, n} := \text{span}_{\mathbb{C}} \{ \bar z^r z^j \mid 0 \leq r \leq q-1, 0 \leq j \leq n-1\} \]
equipped with the inner product from the space $L^2\big(e^{-mQ(z)} \diff A(z) \big)$. Here
 $\diff A$ denotes a suitably normalized area measure on 
$\mathbb{C}$. 
For a point $z_0$ belonging to the interior of certain compact set $\mathcal{S}$ 
and satisfying $\Delta Q(z_0) > 0$, 
we define the rescaled coordinates
\[
z= z_0 + \frac{\xi}{\sqrt{m \Delta Q(z_0)}}, \quad w= z_0 + \frac{\lambda}{\sqrt{m \Delta Q(z_0)}}.
\]
The following universality result is proved in the case $q=2$:  
\begin{multline*}
 \frac{1}{m \Delta Q(z_0)}| K_{q,mQ,n}(z, w)| 
e^{-\frac12 mQ(z)
-\frac12 mQ (w)} 
\to  |L^1_{q-1}(|\xi - \lambda|^2)| e^{-\frac12|\xi-\lambda|^2}
\end{multline*}
as $m,n \to \infty$ while $n \geq m-M$ for any fixed $M > 0$, uniformly for $(\xi, \lambda)$ 
in compact subsets of $\mathbb{C}^2$. The notation $L^1_{q-1}$ stands for the 
associated Laguerre polynomial with parameter $1$ and degree $q-1$.  
This generalizes a result of Ameur, Hedenmalm and Makarov concerning analytic polynomials to \emph{bianalytic polynomials}.
We also discuss how to generalize the result to $q > 2$.  
Our methods include a simplification of a Bergman kernel expansion algorithm of 
Berman, Berndtsson and Sj\"ostrand in the one compex variable setting, and extension 
to the context of polyanalytic functions. We also study off-diagonal behaviour of the kernels $K_{q,mQ,n}$.  
\end{abstract}

\maketitle

\addtolength{\textheight}{2.2cm}







\section{Introduction}
\subsection{Notation}
We will write $\partial X$ and $\mathrm{int}(X)$
for the boundary and the interior of a subset $X$ of the complex plane $\C$.
By $\mathrm{1}_X$ we mean the characteristic function of the set $X$. 
We let 
\[
\diff A(z)=\pi^{-1}\diff x \diff y,\quad\text{where}\,\,\,
z=x+\imag y\in\C,
\]
be the normalized area measure in $\C$, and use the standard Wirtinger 
derivatives
\[
\partial_z:=\tfrac12(\partial_x-\imag\partial_y),\,\,\,
\bar\partial_z:=\tfrac12(\partial_x+\imag\partial_y).
\] 
Will will often omit the subscripts if there is no risk of confusion. 
We write $\Delta= \partial \bar \partial$, and it can be observed that this 
equals to one quarter of the usual Laplacian. 
We write $\mathbb{D}$ for the open unit disk, and more generally 
$\mathbb{D}(z, r)$ for the disk with center $z$ and radius $r$. 
Given a Lebesgue measurable function $w: \mathbb{C} \to \mathbb{R}$, we denote  
by $L^2(w)$ the space of measurable functions $\mathbb{C}\to \mathbb{C}$ which are square-integrable
with respect to the measure $w(z) \diff A(z)$. 

\subsection{Spaces of polyanalytic polynomials}
Let $Q:\mathbb{C} \to \mathbb{R}$ be a continuous function satisfying
\begin{equation} \label{Qgrowthcondition}
 Q(z) \geq (1+\epsilon) \log |z|^2, \qquad |z| \geq C 
\end{equation}
for two positive numbers $\epsilon$ and $C$. This function will be referred to as the \emph{weight}. 
We set
\[ \mathrm{Pol}_{q,n} := \text{span}_{\mathbb{C}} \{ \bar z^r z^j \mid 0 \leq r \leq q-1, 0 \leq j \leq n-1 \}, \]
and 
\[
A^2_{q,mQ,n}:=\mathrm{Pol}_{q,n} \cap L^2\big(e^{-mQ(z)} \big).
\]
The space $A^2_{q,mQ,n}$ is a finite dimensional, and thus closed, 
subspace of $L^2(e^{-mQ(z)})$. 
We see that when $m \geq n+q-1$,  the growth condition on $Q$ implies that 
$A^2_{q,mQ,n}$  contains the whole $\mathrm{Pol}_{q,n}$. 

Notice that $A^2_{1,mQ,n}$
consists of analytic polynomials of degree at most $n-1$. 
For a more general $q \geq 1$, 
functions in the spaces $A^2_{q,mQ,n}$ will be called \emph{$q$-analytic polynomials}. 

By a \emph{$q$-analytic function}, we mean a continuous function that satisfies the equation $\bar\partial^q f=0$ in the sense of distribution theory. 
A function which is $q$-analytic for some $q \geq 1$ is called \emph{polyanalytic}.
Obviously, $q$-analytic polynomials form a subclass of $q$-analytic functions. It is also 
easy to see that a $q$-analytic function $f$ can be written as 
\begin{equation} \label{polyanaldecompo}
 f(z) = \sum_{j=0}^{q-1} \bar z^j f_j(z) 
\end{equation}
for some analytic functions $f_j$. 
The decomposition provides a correspondance between $q$-analytic functions and vector-valued analytic functions; 
this connection is explained in \cite{hh2} in more detail. 

The space $A^2_{q,mQ,n}$ 
possesses the reproducing kernel
\[
K_{q,mQ,n}(z,w) = \sum_j e_j(z) \overline{e_j(w)},
\]
where $\{e_j \}$ is any orthonormal basis for $A^2_{q,mQ,n}$.
It is well-known that $K_{q,mQ,n}$ does not depend on the choice of the basis and 
that the reproducing property 
\[ 
p(z) = \int_{\mathbb{C}} p(w) K_{q,mQ,n}(z,w) e^{-mQ(w)} \diff A(w),
\]
holds for all $p \in A^2_{q,mQ,n}$. 


\subsection{Determinantal point processes}
Assuming $m\geq n+q-1$, which implies that $A^2_{q,mQ,n}$ is $nq$-dimensional, we use  $K_{q,mQ,n}$ to define the following probability distribution on $C^{nq}$:
\begin{equation} \label{probdist}
\diff P_{q,mQ,n}(z_1,\dots, z_{nq}) := \Lambda_{q,mQ,n}(z_1,\dots, z_{nq}) \diff A(z_1) \dots \diff A(z_{nq}),
\end{equation}
where 
\begin{equation}
\Lambda_{q,mQ,n}(z_1,\dots, z_{nq})
:= \frac{1}{(nq)!} \det \bigg[ K_{q,mQ,n}(z_i,z_j)\bigg]_{i,j=1}^{nq} e^{-m\sum_{j=1}^{nq} Q(z_j)}.
\end{equation}
This is a particular instance of a so called \emph{determinantal point process}. That $\diff P_{q,mQ,n}$  is a probability measure
follows from standard arguments  (see any book on random matrices, e.g., \cite{Forr}, \cite{mehta}); this depends only on the fact that $K_{q,mQ,n}$ is a kernel 
of a projection to a $nq$-dimensional subspace of $L^2(e^{-mQ})$ . It is customary 
to identify all the $nq$ copies of the complex plane and permutations of the points $z_1,\dots, z_{nq}$, and think the process as a random configuration
of $nq$ unlabelled points in $\mathbb{C}$.   

For $1\leq k \leq nq$, let us define the \emph{$k$-point intensity functions} $\Gamma^k_{q,mQ,n}$  by 
\[
\Gamma^k_{q,mQ,n}(z_1,\dots, z_k):= \frac{(nq)!}{(nq-k)!} \int_{\mathbb{C}^{nq-k}} \Lambda
(z_1,\dots, z_{nq})\diff A(z_{k+1},\dots, z_{nq}).
\]
It is a well-known fact about determinantal point processes that all the intensity functions are easily expressed with the kernel $K_{q,mQ,n}$:
\[
\Gamma^k_{q,mQ,n}(z_1,\dots, z_k) = \det \bigg[ K_{q,mQ,n}(z_i,z_j)\bigg]_{i,j=1}^{k} e^{-m\sum_{j=1}^{k} Q(z_j)}.
\]
The one-point intensity $\Gamma^1_{q,mQ_n}(z)= K_{q,mQ,n}(z,z)e^{-mQ(z)}$ is particularly important, since integrating it over a set $A$ gives the expected 
number of points in $A$. 

The intensity functions are often called correlation functions in the literature, 
and the weighted kernel $K_{q,mQ,n}(z,w)e^{-\frac12mQ(z)-\frac12m Q(w)}$
is referred to as the \emph{correlation kernel} of the process.

\subsection{Weighted potential theory}
To discuss previous results on analytic polynomial kernels, we need to recall some facts from weighted potential theory. 
Let us assume that the weight $Q$ is $C^2$-regular in $\mathbb{C}$ and satisfies the growth condition \eqref{Qgrowthcondition}. 
We define $\mathcal{N}_+$ and $\mathcal{N}_{+,0}$ to be the sets
\[
\mathcal{N}_+:=\big\{w\in\mathbb{C}:\,\Delta Q(w)>0\big\},\quad
\mathcal{N}_{+,0}:=\big\{w\in\mathbb{C}:\,\Delta Q(w)\ge0\big\}.
\]
The \emph{equilibrium weight} $\widehat{Q}$ is defined as 
the largest subharmonic function in $\mathbb{C}$ which is 
$\leq Q$ everywhere and has the growth bound  
\begin{equation} \label{Qhatgrowthcondition}
\widehat{Q}(z)=\log|z|^2+\mathrm{O}(1),
\quad\text{as}\,\,\,|z|\to+\infty.
\end{equation}
The function $\widehat{Q}$ is then subharmonic and 
it follows from general theory for obstacle problems that it is 
$C^{1,1}$ regular (see \cite{HedMak2}). There is also an elementary proof of this fact
due to Berman \cite{berman}. The coincidence set of the obstacle problem 
is 
\[
\mathcal{S}:= \{ Q(z) = \widehat Q(z) \}.
\]
This is a compact set and because $\widehat Q$ is subharmonic, we have $\mathcal S \subset \mathcal{N}_{+,0}$. 
It is known that $\widehat Q$ is harmonic on $\mathbb{C} \backslash \mathcal{S}$. 
The set $\mathcal{S}$ is a central object in weighted potential theory  as it arises as the support of the unique solution 
to the energy minimization problem 
\[%
\min_{\sigma} I(\sigma), 
\]
where 
\begin{equation} \label{continuousenergyfunct}
I(\sigma):= \frac12 \int_{\mathbb{C}^2} \log \frac{1}{|z-w|^2} \diff \sigma(z) \diff \sigma(w) + \int_{\mathbb{C}} Q(w) \diff \sigma (w).
\end{equation}
The infimum is taken over all compactly supported Borel probability measures. Existence and uniqueness in this problem 
is due to Frostman. For more details, see \cite{st}. Note that the functional \eqref{continuousenergyfunct} 
coincides with the standard logarithmic energy in the special case $Q=0$. 
The solution to energy minimization problem can be written as  (see \cite{HedMak2} or the earlier preprint \cite{HedMak1})
\begin{equation} \label{equimeasure}
\diff \widehat \sigma = \Delta \widehat Q \diff A = \Delta Q \mathrm{1}_{\mathcal{S}} \diff A.
\end{equation} 
The measure $\widehat \sigma$ will be referred to as the \emph{equilibrium measure}.

\subsection{Bergman kernels for analytic polynomials} 
Kernels $K_{1,mQ,n}$ and associated probability densities $\Lambda_{1,mQ,n}$
were studied by Ameur, Hedenmalm and Makarov in \cite{ahm1}, \cite{ahm2} and 
\cite{ahm3}.  
There are three possible interpretations for these point processes: in terms of 
Coulomb gas, free fermions or eigenvalues of random normal matrices. 
For more details, we refer the reader to \cite{zabro1}. 

As $m,n\to+\infty$ while $n=m+\mathrm{o}(n)$, 
 Hedenmalm and Makarov (\cite{HedMak1}, \cite{HedMak2}) showed, building 
on the work of Johansson \cite{joh}, that 
\begin{equation} \label{analyticcorrelationconv}
\frac{(n-k)!}{n!} \Gamma^k_{1,mQ,n} \diff A^{\otimes k} \to \diff \widehat \sigma^{\otimes k}
\end{equation}
in the weak-* sense of measures for all fixed $k$.\footnote{This result is just a special case of their theorem, 
which holds for Coulomb gas models in arbitrary temperatures.}
Here, the notation $\mu^{\otimes k}$ stands for the $k$'th 
tensor power of the measure $\mu$.  As the authors showed in \cite{HedMak2},
this result can be used to show that for any bounded continuous function $g$, 
we have the convergence
\begin{equation} \label{linearstatconv}
\frac{1}{n} \sum_{j=1}^n g(\lambda_n) \to \int_{\mathcal{S}} g(w) \diff \widehat \sigma(w), 
\quad n, m \to \infty, n=m+ \mathrm{o}(n).
\end{equation}
in distribution. Recalling \eqref{equimeasure}, the intuitive interpretation of this is that the points from the determinantal process defined via $K_{1,mQ,n}$ 
tend to accumulate on $\mathcal{S}$ with density $\Delta Q$ as $m,n\to+\infty$ while $n=m+\mathrm{o}(n)$. 
The set $\mathcal{S}$ was called
\emph{droplet} for this reason. 

Let us write \emph{bulk} for the interior of the set $\mathcal{S} \cap \mathcal{N}_+$.  The results 
of \cite{ahm2} show that for a bulk point
$z$ and a $n$-tuple $(z_1,\dots, z_n)$ picked from the density 
\[
\Lambda_{1,mQ,n}(z_1,\dots, z_n)\diff A(z_1)\dots \diff A(z_n), 
\]
the local blow-up process at $z$ with coordinates 
\[
\xi_j:=m^{1/2}[\Delta Q(z)]^{1/2}(z_j-z),
\]
converges to the Ginibre$(\infty)$
process, as $m,n\to+\infty$ while $n=m+\mathrm{o}(1)$. If we write 
$\widetilde \Gamma^k_{1,mQ}$ for the $k$-point intensity of the blow-up process, this means that 
\begin{equation} \label{blowupcorrconv}
\widetilde{\Gamma}^k_{1,mQ,n}(\xi_1,\dots, \xi_k) 
\to \det \big[ e^{\xi_i \bar \xi_j} \big]_{i,j=1}^k e^{-\sum_{i=1}^k |\xi_i|^2}
\end{equation}
as $n, m \to \infty$ while $n=m+ \mathrm{o}(1)$, for all $k$ and $(\xi_1,\dots, \xi_k) \in \mathbb{C}^k$. 
Notice that the kernel $e^{\xi\bar\eta}$
appearing in the determinant on the right hand side is the reproducing kernel 
 of the 
Bargmann-Fock space. 

This fact could be interpreted as a universality result in the spirit of random 
matrix theory and related fields. In general, universality means that there   
exists a scaling limit
which does not depend on particularities of the model (see \cite{deift2} for more discussion). 
In our setting, this is reflected by the fact that the 
the limiting process is the same for all weights $Q$. 


One can also formulate a universality result 
more directly in terms of the kernel $K_{1,mQ,n}$. To do this, let us define the \emph{Berezin density} 
centered at $z$ as 
\[
\berd^{\langle z \rangle}_{1,mQ,n}(w):=
\frac{|K_{1,mQ,n}(w,z)|^2}{K_{1,mQ,n}(z,z)}e^{-mQ(w)}.
\]
The main theorem in \cite{ahm1} was as follows: 
\begin{thm}  [Ameur, Hedenmalm, Makarov] \label{ahmthm1}
Fix $z \in \mathrm{int} \mathcal{S} \cap \mathcal{N}_+$ and suppose that $Q$ is real-analytic in some 
neighborhood of $z$. Then 
\[
\frac{1}{m \Delta Q(z)} 
\berd^{\langle z \rangle}_{1,mQ,n}(z+ \frac{\xi}{\sqrt{\Delta Q(z) m}}) \to e^{-|\xi|^2},
\quad n,m \to \infty, m= n+ \mathrm{o}(1),
\]
where the convergence holds in $L^1\big( \mathbb{C} \big)$.
\end{thm}
It should be mentioned that Berman \cite{berman2} proved a similar result independently, 
also in a higher-dimensional setting. 

Our main result will be a generalization of a slight reformulation of  
theorem \ref{ahmthm1} to the context of more general polyanalytic polynomial kernels $K_{q,mQ,n}$.

\subsection{Polyanalytic Ginibre ensembles}
In a joint paper with Hedenmalm \cite{hh}, we studied kernels $K_{q,mQ,n}$ 
with the weight 
$Q(z)=|z|^2$ and called the associated determinantal point processes 
\emph{Polyanalytic Ginibre ensembles}. As we explained in this paper, 
these point processes describe 
systems of free (i.e. non-interacting) electrons in $\mathbb{C}$ 
in a constant magnetic field of strength $m$ perpendicular to the plane, so that each of the first $q$ Landau 
levels contains $n$ particles. This model has also been studied in physics literature, 
see Dunne \cite{dunne}. 

The analysis was in terms of the Berezin density, which was
defined as in the case $q=1$:
\[
\berd^{\langle z \rangle}_{q,mQ,n}(w):=
\frac{|K_{q,mQ,n}(w,z)|^2}{K_{q,mQ,n}(z,z)}e^{-mQ(w)}.
\] 
In the macroscopic lenght scales, we showed that, 
\begin{gather*}
\berd^{\langle z \rangle}_{q,mQ,n}(w) \diff A(w) \to \delta_z, \qquad |z|<1 \\
\berd^{\langle z \rangle}_{q,mQ,n}(w) \diff A(w) \to \omega_z, \qquad |z| > 1, 
\end{gather*}
as $n, m \to \infty$ while $|n-m| = \mathrm{O}(1)$.  Here $\delta_z$ stands for the Dirac point mass at $z$ and $\omega_z$ for the harmonic measure at $z$ with respect to the domain 
$\mathbb{C} \backslash \overline{\mathbb{D}}$. Notice that both limits are clearly
independent of $q$. 

For microscopic length scales in the bulk $\{ |z| < 1\}$, we obtained 
\begin{equation} \label{polyginibulk}
\berd^{\langle z \rangle}_{q,mQ,n} \big(z + \frac{\xi}{\sqrt{m}}\big) \to \frac{1}{q}L^{1}_{q-1}(|\xi|^2)^2 e^{-|\xi|^2},  \qquad m,n \to \infty, m=n + \mathrm{O}(1)
\end{equation}
where $L^1_{q-1}$ is the associated Laguerre polynomial with parameter $1$ and degree $q-1$. 

\begin{figure} 
\includegraphics[angle= 0,width=80mm]{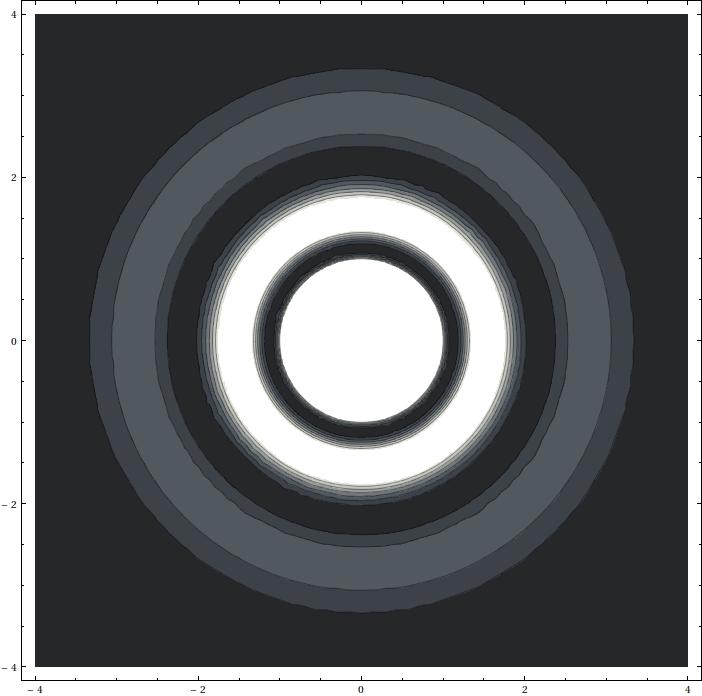} 
\caption{The limiting Berezin density for polyanalytic Ginibre process 
with $q=3$ exhibiting a Fresnel-type ring pattern. Here, white is high and
black is low Berezin density.}
\label{fig-2}
\end{figure} 

It might be interesting to recall that the Laguerre polynomial $L^1_{q-1}$ has $q-1$
zeros on the positive real axis. In terms of the points process, this means that around each point $z$ from the process, 
there are $q-1$ rings around $z$,  the radii of which are of the order of magnitude $m^{-1/2}$, so that there is no repulsion between 
$z$ and points on those rings in the limit as $m, n \to \infty$. So, one would expect the electrons to accumulate on those circles around 
any given electron.  It should be emphasized that this phenomenon is not present 
in the analytic case $q=1$. 

We want to mention here that in \cite{hh}, we also studied  the Berezin kernels near the edge $\{ |z|=1 \}$. We will not discuss such issues in this paper, and will only 
refer the interested reader to the original article. 
  

\subsection{Main results}
Our aim is to generalize the results from polyanalytic Ginibre ensembles to more general weights $Q$. For simplicity, we mostly work 
with $q=2$, but the proof methods should work for any $q$. 
We start by analyzing the one point function $\Lambda^1_{2,mQ,n}(z)$. In section \ref{onepointsection}, 
it is shown in that this expression tends 
to zero for $z \in \mathbb{C} \backslash \mathcal{S}$ in an exponential rate as $m,n \to \infty$. This means 
that the points tend to accumulate on $\mathcal{S}$ as in the analytic case $q=1$. 

Within $\mathcal{S}$, the following theorem presents more detailed information. It states that the bulk scaling limits obtained for polyanalytic Ginibre ensembles are universal. 

\begin{thm} \label{blowupthm}
Set $q=2$. Fix $z_0 \in \mathrm{int} \mathcal{S} \cap \mathcal{N}_+$ and $M > 0$. Assume that $Q$ is $C^2$-smooth, 
satisfies the growth condition \eqref{Qgrowthcondition} and is real-analytic 
in a neighborhood of $z_0$. Then, there exists a number $m_0$ such that for all $m \geq m_0$, we have  
\begin{multline}
\bigg| \frac{1}{m \Delta Q(z_0)} K_{q,mQ,n}\bigg(z_0+\frac{\xi}{\sqrt{m \Delta Q(z_0)}}, 
z_0+\frac{\lambda}{\sqrt{m \Delta Q(z_0)}} \bigg) \bigg| \\
\times 
e^{- \frac12 mQ\big( z_0 + \frac{\xi}{\sqrt{m \Delta Q(z_0)}} \big)
-\frac12 mQ \big(z_0 + \frac{\lambda}{\sqrt{m \Delta Q(z_0)}} \big)} 
= |L^1_{q-1}(|\xi - \lambda|^2)| e^{-\frac12|\xi-\lambda|^2} + \mathrm{O}\big(m^{-1/2}\big),
\end{multline}
as $m \to \infty$ and $n \geq m-M$. The convergence is 
uniform on compact sets of $\mathbb{C}^2$.  
\end{thm}
One can check that with the weight $Q(z)=|z|^2$, the theorem is a slight reformulation 
of \eqref{polyginibulk}. 

Basic structure of our argument will be the same as in \cite{ahm1}. There, the authors relied on two main techniques: 
algorithm of Berman-Berndtsson-Sj\"ostrand \cite{bbs}
to compute asymptotic expansions for Bergman kernels, 
and H\"ormander's $\bar \partial$-estimates. First, we 
will simplify the method of Berman-Berndtsson-Sj\"ostrand (in the one 
complex variable context only) and then extend it to polyanalytic functions. In a joint
paper with Hedenmalm \cite{hh2}, we already showed how to obtain asymptotic expansions 
in the polyanalytic setting, but 
the approach we will take here will be more elementary and also computationally simpler.
  
Whereas in \cite{ahm1} certain estimates for the $\bar \partial$-operator are 
used, we need similar results for the operators $\bar \partial^q$ with $q > 1$. As 
a consequence, we also obtain an off-diagonal decay estimate for bianalytic Bergman kernels,  which, informally speaking, says that correlations are short range in $\mathcal{S}$. Comparing with the 
similar result in \cite{ahm1}, one sees that the decay is essentially as strong as in the case $q=1$. Again, we present a proof for the case $q=2$, but the method should generalize to 
any $q \geq 2$.   
\begin{thm} \label{offdiagonalthm}
Suppose that $Q$ is $C^2$-smooth. Fix a compact set $\mathrm{K}$ in the interior of $\mathcal{S} \cap \mathcal{N}_+$  and constant $M > 0$. Set 
\[
r_{0, \mathrm{K}} := \frac14 \mathrm{dist}\big(\mathrm{K}, \mathbb{C} \backslash (\mathcal{S} \cap \mathcal{N}_+)\big).
\]
Then, there exist positive constants $C,\epsilon$ and $m_0$ such that 
for any $z_0 \in \mathrm{K}$ and $z_1 \in \mathcal{S}$, 
it holds
\[
|K_{2,mQ,n}(z_0,z_1)|^2e^{-mQ(z_0)-mQ(z_1)} 
\leq 
Cm^2 e^{-\epsilon \sqrt{m} \min\{ r_{0,\mathrm{K}}, |z_0-z_1| \}}
\]
where we assume $m \geq \max \{m_0, M-1 \}$ and  $n \geq m-M+1$.  
The constants $C, \epsilon$ and $m_0$ only depend on $Q, \mathrm{K}$ and $M$. 
\end{thm}
One can of course ask what happens if the point $z_1$ is allowed to be outside $\mathcal{S}$. 
The answer will be provided in section \ref{onepointsection}, where we show that even stronger 
decay holds as $m \to \infty$. 

It is possible that this off-diagonal decay estimate, or a variant of it, could also be used in other contexts to extend 
known results about analytic functions to polyanalytic functions.
We should mention at least the work of Ortega-Cerd\`{a} and Ameur \cite{ameurortega} 
concerning Fekete points 
as well as that of Ortega-Cerd\`{a} and Seip \cite{ortegaseip} on description of sampling and interpolation sets.  The latter 
topic in spaces of polyanalytic functions is related to time frequency analysis (see Abreu \cite{abreu}). 

It would also be natural to study asymptotics of $K_{q,mQ,n}$ near the edge of the droplet but 
this question remains open even in the case $q=1$. 

\subsection{Further questions: letting $q$ tend to infinity}
It is also possible to let all the parameters $q,m$ and $n$ tend to infinity in our model.
We explained already in \cite{hh} that the rescaled Berezin transform 
\[
\frac{1}{mq} \berd^{\langle z_0 \rangle}_{q,m|z|^2, n}\big( z_0 + \frac{\xi}{\sqrt{mq}}\big) \qquad |z_0| < 1
\]
converges to the limit $J^1(2 |\xi|)/(2 |\xi|)$ if we first let $n,m \to \infty$ while $n=m + \mathrm{O}(1)$ and then let $q \to \infty$ afterwards. 
Here,  $J^1$ is the standard Bessel function.  Because of theorem \ref{blowupthm}, a
similar result also holds when the weight is more general. Interestingly, this Bessel kernel could be viewed 
as a two-dimensional analogue of the sine kernel from Hermitian random matrix theory: the latter is the Fourier transform 
of a charateristic function of an interval while the former is the Fourier transform of a characteristic
 function of a disk.  

It would interesting to study the asymptotics for $K_{q,mQ,n}$ as $q$ and $n$ go tend to infinity simultaneously. 
To get a very symmetric model, one could set $n=m$, require that $n+q \leq N$ and then let $N$ tend to infinity. It seems 
likely that the above Bessel kernel would also arise here in the limit. 

\subsection{Further questions: fluctuations} 
In \cite{ahm2} and \cite{ahm3}, fluctuation field of the random normal matrix model was shown to converge to Gaussian free field, in the first paper 
with the restriction 
that the test function is supported in the interior of $\mathcal{S}$. The methods of the first paper
should apply to the polyanalytic setting, given the technology we develop in the present paper. The argument of 
\cite{ahm3}, where 
the case of more general test functions was treated using so called Ward identities, seems harder to generalize. 

\section{Construction of local polyanalytic Bergman kernels} \label{localkernelsection}
In this section, we present an algorithm to compute asymptotic expansions for 
polyanalytic Bergman kernels near the diagonal. For analytic functions, 
this is a well-studied topic in several complex variables literature 
see e.g. \cite{cat}, \cite{zel}, \cite{tian}, \cite{mm}. The algorithm we present here is based on the work 
of Berman, Berndtsson and Sj\"ostrand \cite{bbs}, whose method relies on a certain technique from microlocal analysis
(for an detailed exposition 
in the one complex variables setting, see \cite{ahm1}). Here we will show that at least in the one-dimensional case, 
this technique can be dispensed with. In particular, 
we get an alternative way to obtain results of Ameur, Hedenmalm and Makarov 
in the case $q=1$. We then show how to extend this modified algorithm to polyanalytic functions. This provides 
a simplification of the method of \cite{hh2}, 
which was based on the original microlocal analysis technique. 
 
We take an arbitrary $z_0 \in \mathcal{N}_+$ 
and assume that $Q$ is real-analytic in a neighborhood of $z_0$. 
We will also pick $r>0$ such that the following 
conditions are satisfied:
\begin{enumerate}
\item $Q$ is real-analytic in $\mathbb{D}(z_0,r)$ and $\Delta Q(z) > \epsilon > 0 $ on $\mathbb{D}(z_0,r)$. 
\item 
There exists a local polarization 
of $Q$ in $\mathbb{D}(z_0,r)$, i.e. a function $Q: \mathbb{D}(z_0,r) \times  \mathbb{D}(z_0,r) \to \mathbb{C}$
which is analytic in the first and anti-analytic in the second variable, and satisfies 
$Q(z,z) = Q(z)$. 
\item
For $z, w \in \mathbb{D}(z_0, r)$,  we have $\partial_z \bar \partial_w Q(z,w) \neq 0$ and $\bar \partial_w \theta(z,w) \neq 0$; 
these conditions are made possible by condition $(1)$. Here $\theta$ is the phase function which is defined below. 
\item
Taylor expansion of $Q(z, w)$ gives 
\begin{equation} \label{Qtaylor1}
2 \mathrm{Re} Q(z,w) - Q(w)- Q(z) = -\Delta Q(z) |w-z|^2 + \mathrm{O}(|z-w|^3). 
\end{equation}
For details, see p. 1555 in \cite{ahm1}. We require that for $z,w \in \mathbb{D}(z_0,r)$, 
\begin{equation} \label{Qtaylor2}
 2\mathrm{Re} Q(z,w) - Q(w)- Q(z) \leq -\frac12 \Delta Q(z_0) |w-z|^2. 
\end{equation}
\end{enumerate} 
We define the \emph{phase function} $\theta: \mathbb{D}(z_0,r) \times \mathbb{D}(z_0,r)$ as 
\[ \theta(z,w) = \frac{ Q(w) - Q(z,w)}{w-z}. \]
Notice that $\theta$ is analytic in the first variable and real-analytic in the second. 
It can be analytically continued to the diagonal, and we have $\theta(z, z) = \partial_z Q(z)$. More generally, 
\begin{equation}
\theta(z,w) = \frac{Q(w, w)- Q(z,w)}{w-z} = \sum_{j=0}^{\infty} \frac{1}{(j+1)!} (w-z)^j \partial_z^{j+1} Q(z,w).
\label{thetataylor}
\end{equation}
This leads to the following Taylor expansion, which we will need repeatedly: 
\begin{equation}
\bar \partial_w \theta(z,w) = b(z,w) + \frac12 (w-z) \partial_z b(z,w) + 
\frac16(w-z)^2 \partial_z^2 b(z,w) + \dots
\label{dbarthetataylor}
\end{equation}
where 
\[
b(z,w):= \bar \partial_w \partial_z Q(z,w).
\] 

We fix $\chi_0:\mathbb{R} \to \mathbb{R}$ to be a smooth and non-negative cut-off 
function which equals $1$ on the interval $(-\frac23, \frac23)$ and is 
supported on $(-1,1)$. We then define $\chi(z) = \chi_0(\frac{|z-z_0|}{r})$; this function will then be supported on 
$\mathbb{D}(z_0, r)$ and equal $1$ on $\mathbb{D}(z_0, \frac23 r)$.

We will write $A^2_{q,mQ}$ for the subspace of $q$-analytic functions in $L^2(e^{-mQ})$:
\begin{equation}
A^2_{q,mQ}: = \{ f: \mathbb{C} \to \mathbb{C} \mid \bar \partial^q f = 0, 
\| f \|^2_{m} := \int_{\mathbb{C}} |f|^2 e^{-mQ} \diff A < \infty \}.
\end{equation}
Clearly, the spaces $A^2_{q,mQ,n}$ are closed subspaces of $A^2_{q,mQ}$. We will 
use the notation $\| \cdot \|_m$ for the norm in the spaces $L^2(e^{-mQ(z)})$, 
$A^2_{q,mQ}$ and $A^2_{q,mQ,n}$. 

We start by proving a lemma that will be frequently used in the sequel. 
\begin{lem} \label{neglilemma}
Let $m \geq 1$ and $A_m: \mathbb{D}(z_0,r) \times \mathbb{D}(z_0,r) \backslash \{ (z,w) \in \mathbb{C}^2 : z=w \} 
\to \mathbb{C}$ 
be $q$-analytic 
in the first variable and real-analytic in the second. We also assume that there exists 
$K > 0 $ and a positive integer $N$ such that 
\[
|\bar \partial_w^k A_m(z,w)| \leq m^N K, \qquad  
0 \leq k \leq q-1, 
\]
for all $z \in \mathbb{D}(z_0, \frac13 r)$ and   
$w \in \mathbb{D}(z_0, r) \backslash \mathbb{D}(z_0, \frac23 r)$.  
Then, given  integers $k$ and $l$ satisfying $0 \leq k \leq q-1$ and $l \geq 1$, 
there exists $\delta >0$ such that 
\begin{equation}
\int_{\mathbb{D}(z_0,r) \backslash \mathbb{D}(z_0, \frac23 r)} 
\bar \partial_w^k u(w) \bar \partial_w^l \chi(w) A_m(z,w) e^{m(z-w)\theta(z,w)} \diff A(w) 
= \mathrm{O}(\|u\|_m e^{\frac12m Q(z)} e^{-\delta m})
\end{equation}
for any $u \in A^2_{q,mQ}$ and $z \in \mathbb{D}(z_0, \frac13 r)$. The number $\delta$ and the constant
of the error term are independent of $u, m$ and $z$. 
\end{lem} 
\begin{proof}
Let us start with $k=0$. By \eqref{Qtaylor2}, we have 

\begin{multline}
\bigg| \int_{\mathbb{D}(z_0,r)\backslash \mathbb{D}(z_0, \frac23 r)} 
u(w) \bar \partial_w^l \chi(w) A(z,w) e^{m(z-w)\theta(z,w)} \diff A(w)\bigg| \\
\leq m^N K e^{\frac12 m Q(z)} \int_{\mathbb{D}(z_0, r) \backslash \mathbb{D}(z_0, \frac23r)} \big|u(w)
\bar \partial^l_w \chi(w)\big| 
 e^{-\frac12 mQ(w)- \frac14 m \Delta Q(z_0) |w-z|^2} \diff A(w) \\
\leq m^N Ke^{\frac12 m Q(z)} \|u\|_m \bigg[ \int_{\mathbb{D}(z_0, r) \backslash \mathbb{D}(z_0, \frac23r)}
\big| \partial^l_w \chi(w)\big|^2e^{- \frac12 m \Delta Q(z_0) |w-z|^2} \diff A(w) \bigg]^{\frac12} \\
\leq m^N K e^{\frac12 m Q(z)} \|u\|_m e^{- \frac{1}{36} m \Delta Q(z_0) r^2} 
\bigg[\int_{\mathbb{D}(z_0, r) \backslash \mathbb{D}(z_0, \frac23r)}\big| \partial^l_w \chi(w)\big|^2 \diff A(w)
\bigg]^{\frac12} \\
\leq 
C e^{\frac12 m Q(z)}\|u\|_m e^{-\delta m} 
\bigg[\int_{\mathbb{D}(z_0, r) \backslash \mathbb{D}(z_0, \frac23r)}\big| \partial^l_w \chi(w)\big|^2 \diff A(w)
\bigg]^{\frac12},
\end{multline} 
for some positive constants $C$ and $\delta$. This shows the desired statement for the case $k=0$. 

For $k >0$, we integrate by parts: 
\begin{multline}
\int_{\mathbb{D}(z_0,r) \backslash \mathbb{D}(z_0, \frac23 r)} 
\bar \partial_w^k u(w) \bar \partial_w^l \chi(w) A(z,w) e^{m(z-w)\theta(z,w)} \diff A(w) \\
= (-1)^k \int_{\mathbb{D}(z_0,r) \backslash \mathbb{D}(z_0, \frac23 r)} 
 u(w) \bar \partial_w^k \bigg[ \bar \partial_w^l \chi(w) A(z,w) e^{m(z-w)\theta(z,w)} \bigg] \diff A(w). 
\end{multline}
The statement follows after carrying out the differentiation and analyzing each term in the 
resulting sum as in the case $k=0$. 
\end{proof}

Next, we will prove    
an approximate reproducing identity for polyanalytic functions. For the case $q=2$, this was already done
in \cite{hh2}. Here we present the argument 
for general $q \geq 1$. 

\begin{prop} \label{reproprop}
There exists $\delta > 0$, independent of $m$, such that 
for all $z \in \mathbb{D}(z_0, \frac13 r)$ and $u \in A^2_{q,mQ}$, we have
\begin{equation} \label{Rqmapprorepro}
u(z) =  \int_{\mathbb{D}(z_0,r)} u(w) \chi(w) R_{q,m}(z,w) e^{m(z-w)\theta(z,w)} \diff A(w) 
+ \mathrm{O}(\| u \|_{m} e^{mQ(z)/2 - \delta m}), 
\end{equation}
where
\begin{multline}  \label{Rqmformula1}
R_{q,m}(z,w) = m \sum_{s=k}^{q-1} 
\frac{q! (-1)^k}{(q-1-k)! (k+1)! k!}
(\bar z - \bar w)^k \bar \partial_w^k \big( \bar \partial_w \theta e^{m(z-w) \theta} \big)
e^{-m(z-w) \theta}. \\
\end{multline}
The constant of the error term in \eqref{Rqmapprorepro} is independent of $u,z$ and $m$. 
\end{prop}
\begin{proof}
We will use 
the fundamental solution $\frac{1}{(q-1)!} \frac{\bar w^{q-1}}{w}$ of the operator $\bar \partial_w^q$ 
(recall that our reference measure $\diff A$ is the usual area measure divided by $\pi$, so there is no 
need for that normalization here). Because of 
lemma \ref{neglilemma} and $q$-analyticity of $u$, we have
\begin{multline} 
\frac{1}{(q-1)!} \int_{\mathbb{C}} u(w) \chi(w) \bar \partial_w^q \bigg( \frac{(\bar w - \bar z)^{q-1}}{w-z} 
e^{m(z-w)\theta} \bigg) \diff A(w) \\ 
=\frac{(-1)^q}{(q-1)!} \int_{\mathbb{C}} \bar \partial_w^q \big( u(w) \chi(w)\big) \bigg( \frac{(\bar w - \bar z)^{q-1}}{w-z} 
e^{m(z-w)\theta} \bigg) \diff A(w) \\
= \frac{(-1)^q}{(q-1)!} \sum_{k=0}^{q-1} {q \choose k} \int_{\mathbb{D}(z_0,r) \backslash 
\mathbb{D}(z_0, \frac23 r)}
 \bar \partial_w^k u(w) \bar \partial_w^{q-k} \chi(w) 
\frac{\bar w - \bar z}{w-z} 
e^{m(z-w)\theta} \diff A(w) \\
= \mathrm{O}(\| u \|_{m} e^{mQ(z)/2 - \delta m})
\end{multline}
for some $\delta > 0$.

Denoting by $\delta_z(w) \diff A(w)$ the Dirac point mass at $z$, we get
\begin{multline}
u(z) = \int_{\mathbb{C}} u(w) \chi(w) \bigg[ \delta_z(w) - \frac{1}{(q-1)!}
\bar \partial_w^q \bigg( \frac{(\bar w - \bar z)^{q-1}}{w-z} 
e^{m(z-w)\theta} \bigg) \bigg] \diff A(w) + \\ 
\mathrm{O}(\|u \|_m  e^{\frac12 m Q(z) - \delta m}) \\
= \int_{\mathbb{C}} u(w) \chi(w) R_{q,m}(z,w) e^{m(z-w) \theta} \diff A(w) 
+ \mathrm{O}(\|u \|_m e^{\frac12 m Q(z) - \delta m}),
\label{bianalrepro}
\end{multline}
where 
\begin{equation}
R_{q,m}(z,w)e^{m(z-w)\theta} \diff A(w) = \bigg[ \delta_z(w) - \bar \partial_w^q \bigg( \frac{1}{(q-1)!} 
\frac{(\bar w - \bar z)^{q-1}}{w-z} e^{m(z-w)\theta} \bigg) \bigg] \diff A(w)
 .
\label{Rqmkernelsingularform}
\end{equation}
Applying Leibniz rule in the sense of distribution theory shows 
that  
\begin{multline} 
R_{q,m}(z,w) e^{m(z-w) \theta} = m \sum_{k=0}^{q-1} 
\frac{q! (-1)^k}{(q-1-k)! (k+1)! k!}
(\bar z - \bar w)^s \bar \partial_w^k \big( \bar \partial_w \theta e^{m(z-w) \theta} \big).  \\
\label{Rqmformula2}
\end{multline}
So actually, the singularity in \eqref{Rqmkernelsingularform} cancels and $R_{q,m}$ is $q$-analytic 
in $z$ and real-analytic in $w$. 
\end{proof}

A function $L_{q,m}: \mathbb{D}(z_0,r) \times \mathbb{D}(z_0,r) \to \mathbb{C}$ 
which is $q$-analytic in the first variable and real-analytic in the second 
will be called \emph{a local $q$-analytic reproducing kernel} $\text{mod}(e^{-\delta m})$ 
if for any $u \in A^2_{q,mQ}$ and $z \in \mathbb{D}(z_0,\frac13 r)$, we have
\begin{equation} 
u(z) =  \int_{\mathbb{D}(z_0,r)} u(w) \chi(w) L_{q,m}(z,w) e^{-mQ(w)} \diff A(w) \\
+ \mathrm{O}(\| u \|_{m} e^{mQ(z)/2 - \delta m}),
\end{equation}
where the constant of the error term can depend on $Q$, $z_0$ and $r$ but not on $u, z$ and $m$. 
Clearly, proposition \ref{reproprop} shows that $R_{q,m}(z,w)e^{mQ(z,w)}$ satisfies this condition. 
We define a local reproducing kernel $\text{mod}(m^{-k})$ similarly, 
just by replacing the factor $e^{-\delta m}$ in the error term by $m^{-k}$. 
If a local reproducing kernel $L_{q,m}$ with any error term is 
$q$-analytic in $\bar w$ (i.e. satisfies $\partial_w^q L_{q,m}(z,w) =0$), it will be called a local $q$-analytic 
Bergman kernel. 

In \cite{hh2}, we presented an algorithm producing local $q$-analytic Bergman kernels $\text{mod} (m^{-k})$ for arbitrary $k$,  based on a microlocal analysis technique of
 Berman-Berndtsson-Sj\"ostrand \cite{bbs} in the analytic case $q=1$. 
We will next 
show that when $q=1$, Taylor expansion and partial integration is enough. After this we show in the case $q=2$,
how this approach 
can be extended to more general polyanalytic functions. Later, in section \ref{blowupsection}, 
we will show that when $z_0 \in \mathcal{S} \cap \mathcal{N}_+$, local bianalytic Bergman kernels actually provide a near-diagonal
approximation of the kernel $K_{2,mQ,n}$ as $m, n \to \infty$.
 
\subsection{Computation of local analytic Bergman kernels}
The aim is to show how to compute local analytic Bergman kernels $\mod(m^{-k-\frac32})$ 
in the form 
\begin{equation}
 \bigg(
m\Lfun_{1,0}(z,w)+ \Lfun_{1,1}(z,w) + m^{-1} \Lfun_{1,2}(z,w)+\dots+m^{-k}\Lfun_{2, k+1}(z,w)
\bigg) e^{mQ(z,w)}, 
\end{equation}
where all the coefficient functions $\Lfun_{1,j}$ are analytic in $z$ and $\bar w$. 
We denote by $X_j(z,w)$ a function on $\mathbb{D}(z_0,r) \times \mathbb{D}(z_0,r)$ 
which is analytic in $z$ and real-analytic in $w$ 
but whose exact form is not of interest to us. The number $\delta$ will denote a positive 
number that can change at each step.  

Let $u \in A^2_{1,mQ}$ and $z \in \mathbb{D}(z_0, \frac13 r)$. We check from \eqref{Rqmformula1} that 
\begin{equation} \label{R1mformula}
R_{1,m}(z,w) = m \bar \partial_w \theta(z,w).
\end{equation}
Then, recalling proposition \ref{reproprop} and 
the Taylor expansion \ref{dbarthetataylor}, we compute 
\begin{multline}
u(z) = m \int_{\mathbb{D}(z_0,r)} u(w) \chi(w) \bigg[ b(z,w) +(w-z)\frac12 \partial_z b(z,w) + 
(w-z)^2 X_1(z,w) \bigg] \\ \times e^{m(z-w)\theta} \diff A(w)
+ \mathrm{O}(\|u\|_m e^{\frac12 m Q(z) - \delta m })  \\
= m \int_{\mathbb{D}(z_0,r)} u(w) \chi(w) b(z,w) e^{m(z-w)\theta} \diff A(w) \\
- \int_{\mathbb{D}(z_0,r)} u(w) \chi(w) \frac{1}{\bar \partial_w \theta} \bigg[ \frac12 \partial_z b(z,w) + (w-z)X_1(z,w) \bigg] 
\bar \partial_w e^{m(z-w) \theta} \diff A(w) \\
= m \int_{\mathbb{D}(z_0,r)} u(w) \chi(w) b(z,w) e^{m(z-w)\theta} \diff A(w) \\ 
+\int_{\mathbb{D}(z_0,r)} u(w) \chi(w) \bar \partial_w \bigg[ \frac{1}{\bar \partial_w \theta} 
\bigg( \frac12 \partial_z b(z,w) + (w-z)X_1(z,w) \bigg) \bigg] e^{m(z-w) \theta} \diff A(w) \\
+\mathrm{O}(\|u \|_m e^{\frac12 m Q(z) - \delta m}), 
\label{analsecondterm}
\end{multline}
where for the last equality, we needed an application of lemma \ref{neglilemma}. Here and 
later in computations of this nature, the choice of $\delta$ and the error constant 
is independent of $u$, $m$ and $z$. 

We Taylor expand $\frac{1}{\bar \partial_w \theta} = \frac1b + (w-z)X_2$ using \eqref{dbarthetataylor},  
and continue the analysis:
\begin{multline} \label{analyticlocalkernel}
u(z) = \int_{\mathbb{C}} u(w) \chi(w) 
\bigg(m b + \frac12 \bar \partial_w \frac{\partial_z b}{b} \bigg) e^{m(z-w)\theta} \diff A(w) \\ 
+ \int_{\mathbb{D}(z_0,r)} u(w) \chi(w) (w-z) X_3 e^{m(z-w) \theta} \diff A(w)  
+\mathrm{O}(\|u\|_m e^{\frac12 m Q(z) - \delta m}) \\
=   \int_{\mathbb{D}(z_0,r)} u(w) \chi(w) 
\bigg(m b + \frac12 \bar \partial_w \frac{\partial_z b}{b} \bigg) e^{m(z-w)\theta} \diff A(w) \\ 
- \frac{1}{m} \int_{\mathbb{D}(z_0,r)} u(w) \chi(w)   X_3 
\frac{1}{\bar \partial_w \theta} \bar \partial_w e^{m(z-w) \theta} 
\diff A(w) + \mathrm{O}(\|u \|_m e^{\frac12 m Q(z) - \delta m}) \\
=   \int_{\mathbb{D}(z_0,r)} u(w) \chi(w) 
\bigg(m b + \frac12 \bar \partial_w \frac{\partial_z b}{b} \bigg) e^{m(z-w)\theta} \diff A(w) \\ 
+ \frac1m \int_{\mathbb{D}(z_0,r)} u(w) \chi(w) \bar \partial_w \bigg[\frac{1}{\bar \partial_w \theta} X_3 \bigg] 
e^{m(z-w) \theta} 
\diff A(w) + \mathrm{O}(\|u \|_m e^{\frac12 m Q(z) - \delta m}) \\
= \int_{\mathbb{D}(z_0,r)} u(w) \chi(w) 
\bigg(m b + \frac12 \bar \partial_w \frac{\partial_z b}{b} \bigg) e^{m(z-w)\theta} \diff A(w) \\ 
+ \mathrm{O}(\|u\|_m e^{\frac12 m Q(z)} m^{-\frac32}).
\end{multline}
For the third equality, lemma \ref{neglilemma} was again used. 
The fact that we get a factor $m^{-3/2}$ in last error term 
is a consequence of a small computation which we include here for the convenience 
of the reader. Let $Y: \mathbb{D}(z_0,r) \times \mathbb{D}(z_0,r) \to \mathbb{C}$
be $C^2$-smooth and assume 
\[
\max_{z, w \in \mathbb{D}(z_0,r)} |Y(z,w)| \leq C
\]
for a constant $C > 0$. Then, using \eqref{Qtaylor2} and Cauchy-Schwarz inequality, 
\begin{multline} \label{computation1}
\frac{1}{m} \bigg| \int_{\mathbb{D}(z_0,r)} u(w)\chi(w) Y(z,w) e^{m(z-w)\theta} \diff A(w) \bigg| \\
\leq \frac{C}{m} \int_{\mathbb{D}(z_0,r)} 
|u(w)|e^{m\mathrm{Re} Q(z,w)- mQ(w)} \diff A(w) \\
\\ \leq \frac{C}{m} \|u\|_m e^{\frac12 m Q(z)}
\bigg[ \int_{\mathbb{D}(z_0,r)}
e^{-\frac12m \Delta Q(z_0)|w-z|^2} \diff A(w) \bigg]^{\frac12} \\
= \frac{C}{m^{3/2}\sqrt{\Delta Q(z_0)}} 
\|u\|_m e^{\frac12 m Q(z)} \bigg[ \int_{\mathbb{D}(z_0,r)}
e^{-\frac12|w-z|^2} \diff A(w) \bigg]^{\frac12} \\
= \mathrm{O}(\|u\|_m e^{\frac12 m Q(z)} m^{-\frac32}).
\end{multline}

The conclusion is that 
\[
 \bigg[m b(z,w) + \frac12 \bar \partial_w \frac{\partial_z b(z,w)}{b(z,w)}\bigg]e^{mQ(z,w)}
\]
is a local analytic Bergman kernel $\mod (m^{-\frac32})$. It is possible 
to continue in the same way and compute local Bergman kernels $\mod(m^{-k-\frac12})$
for any positive integer $k$;  in order to do this, one just has to use higher order Taylor 
expansions of $\bar \partial_w \theta$.  Notice that the local Bergman kernels provided by this 
process are indeed conjugate analytic in $w$; this follows from the fact that the coefficient functions 
in the Taylor expansion \eqref{dbarthetataylor} have this property. 

\begin{rem}
In the computation of local Bergman kernels, we do not necessarily need to require that $u \in A^2_{1,mQ}$: it 
is enough to assume that $u$ is analytic in $\mathbb{D}(z_0,r)$ and then replace $\|u\|_m^2$ by 
$\int_{\mathbb{D}(z_0,r)} |u(w)|^2 e^{-mQ(w)} \diff A(w)$ in the error terms.
\end{rem}

\subsection{Local bianalytic Bergman kernels} \label{bianalyticsection}
We will now explain how to extend the above method to a more general polyanalytic setting. 
We focus on the case $q=2$. The functions satisfying $\bar \partial^2 u=0$ 
will be called \emph{bianalytic}. 
The intention is to show how to compute local bianalytic Bergman kernels $\mod(m^{-k-\frac32})$ 
in the form 
\begin{equation} 
\bigg( m^2 \Lfun_{2,0}(z,w) + m^{1}\Lfun_{2,1}(z,w) + \dots + m^{-k}\Lfun_{2,2+k}(z,w) \bigg) e^{mQ(z,w)},
\label{bilocalkernel}
\end{equation}
where $z,w \in \mathbb{D}(z_0,r)$ and all the 
coefficient functions are bianalytic in $z$ and $\bar w$.
Proceeding as in the case $q=1$, we will expand the kernel $R_{2,m}$ in powers of $w-z$ so that the coefficients 
are bianalytic in $\bar w$. 
The following proposition will replace the partial integration that 
was performed in the analytic setting.  

\begin{prop} \label{partintprop}
Let $A: 
\mathbb{D}(z_0, r) \times \mathbb{D}(z_0, r) \to \mathbb{C}$ be a $C^2$-smooth function. 
Then, there exists $\delta > 0$ such that for any  
$u \in A^2_{q,mQ}$ and $z \in \mathbb{D}(z_0, \frac13 r)$, we have 
\begin{multline} \label{partintlemma}
\int u(w) \chi(w) m^2 (z-w)^2 A(z,w) e^{m(z-w) \theta} \diff A(w) \\
= \int u(w) \chi(w) \bigg[ -\bar \partial^2_w \frac{A(z,w)}{(\bar \partial_w \theta)^2} 
+ 
m(z-w) \bigg(-2\frac{\bar \partial_w A(z,w)}{\bar \partial_w \theta} + 
3 \frac{A(z,w) \bar \partial_w^2 \theta}{(\bar \partial_w \theta)^2} \bigg) 
\bigg]  e^{m(z-w) \theta} \diff A(w) \\
+  \mathrm{O}(\|u \|_m e^{\frac12 m Q(z) - \delta m}).
\end{multline}
The constant of the error term is independent of $u$, $m$ 
and $z$. 
\end{prop}
\begin{proof}
The proof involves only partial integration. 
\begin{multline} \label{partint1}
\int_{\mathbb{D}(z_0,r)} u(w) \chi(w) m^2 (z-w)^2 A(z,w) e^{m(z-w) \theta} \diff A(w) \\
= \int_{\mathbb{D}(z_0,r)} u(w) \chi(w) m(z-w) \frac{A(z,w)}{\bar \partial_w \theta} \bar \partial_w 
e^{m(z-w) \theta} \diff A(w)
\\
= - \int_{\mathbb{D}(z_0,r)} u(w) \chi(w) m(z-w) \bar \partial_w \frac{A}{\bar \partial_w \theta}  e^{m(z-w) \theta} \diff A(w)\\
- \int_{\mathbb{D}(z_0,r)} \bar \partial_w u(w) \chi(w) m(z-w) \frac{A}{\bar \partial_w \theta} e^{m(z-w) \theta} \diff A(w)\\ 
- \int_{\mathbb{D}(z_0,r) \backslash \mathbb{D}(z_0, \frac23 r)}  u(w) \bar \partial_w \chi(w) m(z-w) \frac{A}{\bar \partial_w \theta} e^{m(z-w) \theta} \diff A(w)\\
= - \int_{\mathbb{D}(z_0,r)} u(w) \chi(w) m(z-w) \bar \partial_w \frac{A}{\bar \partial_w \theta}  e^{m(z-w) \theta} \diff A(w) \\
- \int_{\mathbb{D}(z_0,r)} \bar \partial_w u(w) \chi(w) m(z-w) \frac{A}{\bar \partial_w \theta} e^{m(z-w) \theta} \diff A(w) 
+\mathrm{O} \big( \|u \|_m e^{\frac12 m Q(z) - \delta m} \big),
\end{multline}
where we used lemma \ref{neglilemma} to get the last equality. 

We leave the first integral in the last expression of \eqref{partint1} 
as such and continue with the analysis of the second. 
\begin{multline} \label{partint4}
-\int \bar \partial_w u(w) \chi(w) m(z-w) \frac{A}{\bar \partial_w \theta} e^{m(z-w) \theta} 
\diff A(w) \\
= -\int \bar \partial_w u(w) \chi(w) \frac{A}{(\bar \partial_w \theta)^2} \bar \partial_w e^{m(z-w) \theta} 
\diff A(w) \\
= \int \bar \partial_w u(w) \chi(w) \bar \partial_w \bigg( \frac{A}{(\bar \partial_w \theta)^2} 
\bigg)  
e^{m(z-w) \theta} \diff A(w) \\
+\int \bar \partial_w u(w) \bar \partial_w \chi(w) \bigg( \frac{A}{(\bar \partial_w \theta)^2} 
\bigg)  
e^{m(z-w) \theta} \diff A(w)
\\
= \int \bar \partial_w u(w) \chi(w) \bar \partial_w \bigg( \frac{A}{(\bar \partial_w \theta)^2} 
\bigg)  
e^{m(z-w) \theta} \diff A(w) +  \mathrm{O}(\|u \|_m e^{\frac12 m Q(z) - \delta m})  \\
= -\int u(w) \bar \partial_w \bigg[\chi(w) \bar \partial_w \bigg( \frac{A}{(\bar \partial_w \theta)^2} 
\bigg)  
e^{m(z-w) \theta}\bigg] \diff A(w) +  \mathrm{O}(\|u \|_m e^{\frac12 m Q(z) - \delta m}) \\
= 
- \int u(w) \chi(w) \bar \partial_w \bigg[
\bar \partial_w \bigg( \frac{A}{(\bar \partial_w \theta)^2} 
\bigg)  
e^{m(z-w) \theta}\bigg] \diff A(w) +  \mathrm{O}(\|u \|_m e^{\frac12 m Q(z) - \delta m}),
\end{multline}
where the third and the fifth equality depended on lemma \ref{neglilemma}.  
After carrying out the differentiations in the last integrand, the assertion follows 
from combining \eqref{partint1} with \eqref{partint4}. 
\end{proof}
We will now illustrate how the result 
can be used by computing two first 
terms of the expansion \eqref{bilocalkernel}. 
We let $X_j$ stand for a function defined on $\mathbb{D}(z_0,r) \times \mathbb{D}(z_0,r)$
which is bianalytic in the first variable and real-analytic in the second. 

We see from \eqref{Rqmformula1} that 
\begin{equation} \label{R2mformula}
R_{2,m}(z,w) = 2m \bar \partial_w \theta(z,w) - m(\bar z - \bar w) \bar \partial_w^2 \theta(z,w) - m^2
|z-w|^2 \big[\bar \partial_w \theta(z,w) \big]^2.
\end{equation}
By Taylor expanding this expression with \eqref{dbarthetataylor}, we get
\begin{multline}
u(z) = \int u(w) \chi(w) \bigg[ 2m \big(b + \frac12 (w-z) \partial_z b \big) -m (\bar z - \bar w) 
\big( \bar \partial_w b + \frac12 (w-z)\bar \partial_w \partial_z b \big) \\
- m^2|z-w|^2b^2 
+ m^2(z-w)^2 (\bar z - \bar w) b \partial_z b \\
+ m(z-w)^2 X_1(z,w) + m^2 (z-w)^3 X_2(z,w) \bigg] e^{m(z-w)\theta} \diff A(w) \\
+\mathrm{O}\big(\|u\|_m e^{\frac12 m Q(z)} e^{-\delta m}\big).
\label{bianalreproexp}
\end{multline}
We require the coefficients in \eqref{bilocalkernel} 
to be bianalytic in $z$ and $\bar w$, which means in particular 
that they may not contain the factor $(z-w)$ to any degree higher than $1$. Only the last three 
terms in \eqref{bianalreproexp} contain $(z-w)^2$, so they are the only ones which require 
further analysis. 
For the term with $X_1$,  proposition \ref{partintprop} shows that 
\begin{multline}
\int_{\mathbb{C}} u(w) \chi(w) m(z-w)^2 X_1(z,w) e^{m(z-w)\theta} \diff A(w) \\
= \int_{\mathbb{C}} u(w) \chi(w) \bigg( \frac1m X_3(z,w) + (z-w) X_4(z,w) \bigg) 
e^{m(z-w)\theta} \diff A(w) \\
+ \mathrm{O}(\|u \|_m e^{\frac12 m Q(z) - \delta m})
= \mathrm{O}(\|u \|_m e^{\frac12 m Q(z)} m^{-1/2})
\label{X1}
\end{multline}
As we are only interested in the coefficients for $m^2$ and $m$, we conclude that this term is 
negligible for our purposes. Note that the extra factor $m^{-1/2}$ in the error term 
comes from the computation similar to \eqref{computation1}. 
For the term with $X_2$, two applications of proposition \ref{partintprop} are needed to show this: 
\begin{multline}
\int_{\mathbb{C}} u(w) \chi(w) m^2(z-w)^3 X_2(z,w) e^{m(z-w)\theta} \diff A(w) \\
= \int_{\mathbb{C}} u(w) \chi(w) \bigg( (z-w) X_5(z,w) + m(z-w)^2 X_6(z,w) \bigg) 
e^{m(z-w)\theta} \diff A(w)  \\ 
+ \mathrm{O}(\|u \|_m e^{\frac12 m Q(z) - \delta m}) \\
= \int_{\mathbb{C}} u(w) \chi(w) \bigg( (z-w) X_5(z,w)  
+ \frac{1}{m} X_7(z,w) + (z-w) X_8(z,w) \bigg) e^{m(z-w)\theta} \diff A(w) \\ 
+  \mathrm{O}(\|u \|_m e^{\frac12 m Q(z) - \delta m}) 
= \mathrm{O}(\|u \|_m e^{\frac12 m Q(z)}m^{-1/2}).
\end{multline}
We now set 
\[ A_1(z,w) = (\bar z - \bar w) b \partial_z b, \]
and analyze the corresponding term in \eqref{bianalreproexp} with the proposition \ref{partintprop}. 
\begin{multline}
\int u(w) m^2 (z-w)^2 A_1(z,w) e^{m(z-w) \theta} \diff A(w) \\
= \int u(w) \bigg[- \bar \partial^2_w \frac{A_1(z,w)}{(\bar \partial_w \theta)^2} + 
m(z-w) A_2(z,w) 
\bigg] e^{m(z-w) \theta} \diff A(w) \\
+ \mathrm{O}(\|u\|_m e^{\frac12 m Q(z) - \delta m}),
\label{A1}
\end{multline}
where 
\begin{equation}
A_2(z,w) = 
\bigg(- 2\frac{ \bar \partial_w A(z,w)}{\bar \partial_w \theta} + 
3 \frac{A(z,w) \bar \partial_w^2 \theta }{(\bar \partial_w \theta)^2} \bigg).
\label{A2}
\end{equation}
We expand $A_2$ in powers in $z-w$: 
\begin{multline}
A_2(z,w) = 
- 2\frac{ \bar \partial_w A_1(z,w)}{\bar \partial_w \theta} + 
3 \frac{A_1(z,w) \bar \partial_w^2 \theta }{(\bar \partial_w \theta)^2} 
= - 2\frac{ \bar \partial_w A_1(z,w)}{b} + 
3 \frac{A_1(z,w) \bar \partial_w b }{b^2} 
+ (z-w)X_9  \\
= 2 \partial_z b + (\bar z - \bar w) 
\bigg(-2 \bar \partial_w \partial_z b 
+ \frac{ \partial_z b \cdot \bar \partial_w b}{b} \bigg)
+(z-w)X_9.
\label{A2expanded}
\end{multline}
We implement this into \eqref{A1} and see by the same argument as in \eqref{X1} that 
the term with $X_9$ is negligible
We now put everything together within \eqref{bianalreproexp}: 
\begin{multline}
u(z)= \int_{\mathbb{C}} u(w) \chi(w)\bigg[m^2 \Lfun_{2,0} + m\Lfun_{2,1}\bigg] e^{m(z-w) \theta} 
\diff A(w) 
+\mathrm{O} (\|u \|_m m^{-\frac12} e^{\frac12 m Q(z)}) , 
\end{multline}
where 
\begin{equation} \label{b22}
\Lfun_{2,0} = -|z-w|^2 b^2  
\end{equation}
and
\begin{equation} \label{b21}
\Lfun_{2,1} = 2 b + (z-w) \partial_z b - (\bar z- \bar w) \bar \partial_w b 
+ |z-w|^2 \bigg ( -\frac32 \bar \partial_w \partial_z b + \frac{\partial_z b 
\cdot \bar \partial_w b}{b} \bigg). 
\end{equation}

We conclude that 
\begin{multline}
K^{(2)}_{2,m}(z,w) := \bigg\{-m^2|z-w|^2 b^2 \\
+m \bigg[2 b + (z-w) \partial_z b - (\bar z- \bar w) \bar \partial_w b 
+ |z-w|^2 \bigg ( -\frac32 \bar \partial_w \partial_z b + \frac{\partial_z b 
\cdot \bar \partial_w b}{b} \bigg)\bigg] \bigg\} e^{mQ(z,w)}
\end{multline}
is a local bianalytic Bergman kernel $\mod(m^{-1/2})$. 

We have also computed the third term:
\begin{equation} \label{b20}
\Lfun_{2,2} = 2 \bar \partial_w \partial_z \log b + 
(\bar w - \bar z) \bar \partial_w^2 \partial_z \log b
+ (z-w) \bar \partial_w \partial_z^2 \log b
+|z-w|^2 M(z,w), 
\end{equation}
where 
\begin{multline}
M = 
\frac32 \frac{\bar \partial_w^2 \partial_z b \cdot \partial_z b}{b^2} 
-\frac{13}{2} \frac{\partial_z b \cdot \bar \partial_w \partial_z b \cdot \bar \partial_w b}{b^3}
+\frac32 \frac{(\bar \partial_w \partial_z b)^2}{b^2}\\
-\frac{(\partial_z b)^2 (\bar \partial_w^2 b)}{b^3}
+ \frac{17}{4} \frac{(\partial_z b)^2 \cdot (\bar \partial_w b)^2}{b^4}
-\frac23 \frac{\partial_z^2 \bar \partial_w^2 b}{b}
+\frac32 \frac{\partial_z^2 \bar \partial_w b \cdot \bar \partial_w b}{b^2} \\
-\frac{(\partial_z^2 b)(\bar \partial_w b)^2}{b^3}
+\frac13 \frac{\bar \partial_w^2 b \cdot \partial_z^2 b}{b^2}.
\end{multline}
The rather long computations are presented in the appendix. 

\section{Some estimates for the one-point function } \label{onepointsection}
In this section, we provide an apriori bound for the \emph{one point function}
\[
\Gamma^1_{2,mQ,n}(z):=K_{2,mQ,n}(z,z)e^{-mQ(z)}
\]
for $z \in \mathcal{S}$ and show that for 
$z \in \mathbb{C} \backslash \mathcal{S}$ we have convergence to zero with a rate that is exponential in $m$.  

In \cite{ahm1}, analogous results were shown for kernels $K_{1,m,n}$ relying heavily 
on the fact that $\log |f|$ is subharmonic 
whenever $f$ is an analytic function. This is no longer true when $f$ is more general polyanalytic function, so 
we must use a different strategy.

Let us recall the following result, which is just proposition 8.1 of \cite{hh2} in slightly altered form. 
Later, in lemma \ref{pointwiselemma2}, we will show how a more general version of this result follows 
from proposition 4.1 of that paper. 

\begin{lem} \label{pointwiselemma}
For any bianalytic function $u$ and $z \in \mathbb{C}$, 
we have 
\begin{equation} \label{pointwise1}
|u(z)|^2e^{-mQ(z)} \leq m(8+48A^2)e^A \int_{\mathbb{D}(z, m^{-1/2})} |u(w)|^2 e^{-mQ(w)} \diff A(w)
\end{equation}
and 
\begin{equation} \label{pointwise2}
|\bar \partial u(z)|^2e^{-mQ(z)} \leq 3m^2e^A \int_{\mathbb{D}(z, m^{-1/2})} |u(w)|^2 e^{-mQ(w)} \diff A(w),
\end{equation}
where 
\[
A:= \sup_{w \in \mathbb{D}(z, m^{-1/2})} |\Delta Q(w)|
\]
\end{lem} 
As an implication, we get two useful estimates for the 
kernel $K_{2,mQ,n}$. 
Namely, taking $u(w):= K_{2,mQ,n}(w,z)$ for some fixed $z$ and using 
\[
\int_{\mathbb{C}} |K_{2,mQ,n}(w,z)|^2e^{-mQ(w)} \diff A(w) = K_{2,mQ,n}(z,z),
\]
we get the following estimate for the one-point function on $\mathcal{S}$:
\begin{equation} \label{diagonalestimate1}
K_{2,mQ,n}(z,z)e^{-mQ(z)} \leq m (8+48A_{\mathcal{S}}^2)e^{A_{\mathcal{S}}}, \qquad z \in \mathcal{S}, \quad m \geq 1,
\end{equation}
where 
\[
A_{\mathcal{S}}:= \sup_{\dist(w, \mathcal{S}) \leq 1} |\Delta Q(w)|.
\]

Combining \eqref{diagonalestimate1} with Cauchy-Schwarz inequality gives
\begin{multline}\label{apriorioffdiagonal}
|K_{2,mQ,n}(z,w)|^2e^{-mQ(z)-mQ(w)} \\
\leq 
K_{2,mQ,n}(z,z)e^{-mQ(z)} 
K_{2,mQ,n}(w,w)e^{-mQ(w)} \\
\leq m^2 (8+48A_{\mathcal{S}}^2)^2e^{2A_{\mathcal{S}}}, \qquad z,w \in \mathcal{S}.
\end{multline}

We now apply the lemma \ref{pointwiselemma} to provide an analogue 
of lemma 3.5 in \cite{ahm1}. We will need a simple weighted maximum principle
for analytic polynomials provided by lemma 3.4 of the same paper 
(see also  theorem III.2.1 in \cite{st}).
The proof is based on using the definition of $\widehat Q$ 
and an application of maximum principle. 
\begin{lem} \label{weightedmaxpri}
Suppose $m \geq 1$ and $n \leq m+1$. 
Let $u$ be an analytic polynomial of degree $\leq n-1$ satisfying
\[ 
|u(z)|^2e^{-mQ(z)} \leq 1, \qquad z \in \mathcal{S}.
\]  
Then 
\[
|u(z)|^2e^{-m\widehat{Q}(z)} \leq 1, \qquad z \in \mathbb{C}. 
\]
\end{lem}

\begin{prop} \label{bimaxprinciple}
Suppose $m \geq 1$ and $n \leq m$.
Then, there exists a constant depending only on $Q$ such that
for all $u \in A^2_{2,mQ,n}$ and $z \in \mathbb{C}$, it holds 
\[ |u(z)|^2 \leq Cm^2 \|u\|^2_{mQ}e^{m \widehat Q(z)}, \qquad m \geq 1. \]
\end{prop}
\begin{proof}
We take $z_0 \in \mathbb{C}$ which minimizes the quantity 
\[ \max_{z \in \mathcal{S}} |z- z_0|. \]
Let $c$ be the value of this expression corresponding to the optimal choice of $z_0$.
We write $u(z) = p(z) + (\bar z - \bar z_0) q(z)$ for analytic $p$ and $q$. According to \eqref{pointwise2}, 
\[ |(\overline{z - z_0})q(z)|^2e^{-mQ(z)} \leq 3m^2e^{A_{\mathcal{S}}} c^2\|u\|^2_m, \qquad z \in \mathcal{S}, \]
and 
\begin{multline}
|p(z)|^2e^{-mQ(z)} \leq 2\big[ |u(z)|^2 + |z-z_0|^2 |q(z)|^2 \big]e^{-mQ(z)}
\\
\leq 2\big[ (8+48A_{\mathcal{S}}^2)e^{A_{\mathcal{S}}} m + 3c^2 e^{A_{\mathcal{S}}}m^2 \big] \|u\|_m^2, 
\qquad z \in \mathcal{S},
\end{multline}
where 
\[ 
A_{\mathcal{S}} := \sup_{\dist(z, \mathcal{S}) \leq 1} |\Delta Q(z)|.
\]
We can now apply lemma \ref{weightedmaxpri} to get  
\begin{equation} \label{maxpri1}
 |(\overline{z - z_0})q(z)|^2e^{-m\widehat Q(z)} \leq me^{A_{\mathcal{S}}} c^2\|u\|^2_m, \quad z \in \mathbb{C}
\end{equation}
and 
\begin{equation} \label{maxpri2}
|p(z)|^2e^{-m\widehat Q(z)} \leq
 2\big[ (8+48A_{\mathcal{S}}^2)e^{A_{\mathcal{S}}} m + 3c^2 e^{A_{\mathcal{S}}}m^2 \big] \|u\|_m^2, \quad z \in \mathbb{C}.
\end{equation}
The statement of the proposition follows from putting \eqref{maxpri1} and \eqref{maxpri2} together:
\begin{equation}
|u(z)|^2 e^{-m\widehat{Q}(z)} \leq 2|p(z)|^2e^{-m\widehat{Q}(z)}+2|(z-z_0)|^2|q(z)|^2e^{-m\widehat{Q}(z)} 
\leq Cm^2 \|u\|_{m},
\end{equation}  
where the constant $C$ only depends on $A_{\mathcal{S}}$ and $c$. 
\end{proof}

From this result, we deduce
\begin{multline}
K_{2,mQ,n}(z,z)^2e^{-mQ(z)} \leq Cm^2 \|K_{2,mQ,n}(\cdot, z) \|^2_m e^{-m(Q(z)-\widehat Q(z))},  \qquad n \leq m \\
\end{multline}
and because 
\[
\|K_{2,mQ,n}(\cdot, z) \|_m^2 = K_{2,mQ,n}(z,z),
\]
we get
\begin{equation}\label{onepointest11}
K_{2,mQ,n}(z,z)e^{-mQ(z)} \leq Cm^2 e^{-m(Q(z)-\widehat Q(z))}, \qquad n \leq m.
\end{equation}
for all $z \in \mathbb{C}$. Notice that this estimate is only interesting for $z \in \mathbb{C} \backslash \mathcal{S}$, 
because for $z \in \mathcal{S}$, we already have a better estimate in \eqref{diagonalestimate1}. 
We conclude that the one-point function $K_{2,mQ,n}(z,z)e^{-mQ(z)}$ decays 
exponentially to zero for a fixed $z \in \mathbb{C} \backslash \mathcal{S}$ as $m \to \infty$. Moreover, 
the growth conditions \eqref{Qgrowthcondition} and \eqref{Qhatgrowthcondition} imply  
that for any neighborhood $\mathcal{D}$ of $\mathcal{S}$, we have 
\[
\int_{\mathbb{C} \backslash \mathcal{D}} K_{q,mQ,n}(z,z) e^{-mQ(z)}\diff A(z) \to 0, \qquad m \to \infty, n \leq m.
\]
Finally, we want to record the the following off-diagonal analogue of \eqref{onepointest11}:
\begin{multline}
|K_{2,mQ,n}(w,z)|^2e^{-mQ(w)-mQ(z)}
\leq Cm^2K_{2,mQ,n}(z,z)e^{-mQ(z)}e^{-m(Q(w)-\widehat Q(w))} \\
\leq C^2m^4 e^{-m(Q(z)-\widehat Q(z))}e^{-m(Q(w)-\widehat Q(w))}, \quad n \leq m.
\end{multline}
In particular, for fixed $z$ and $w$, we have again exponential decay to $0$ whenever one of the two points 
is in $\mathbb{C} \backslash \mathcal{S}$. This should be compared with the off-diagonal damping theorem 
\ref{offdiagonalthm}, which deals with the case when both 
$z$ and $w$ belong to $\mathcal{S}$.

\section{H\"ormander-type estimates for $\bar \partial^2$ } \label{hormandersection}
The purpose of this section is to prove theorem \ref{rhoprop}, 
which is an estimate for a solution of the equation $\bar \partial^2 u = f$,
where a certain growth condition is imposed on the solution near infinity.  
This result will be used in section \ref{offdiagonalsection} to prove an off-diagonal decay estimate 
for bianalytic Bergman kernels. 

We will assume 
that $\widehat Q \geq 1$. We can always 
arrange this by adding a sufficiently big constant to $Q$. This just means 
that the corresponding reproducing kernel will be multiplied by a constant and 
so the problem remains essentially unchanged. 
 
We start by fixing  $z_0 \in \mathrm{int} \mathcal{S}\cap \mathcal{N}_+$ and setting 
\begin{gather*}
r_0 = \frac14 \text{dist}\big(z_0, \mathbb{C} \backslash (\mathcal{S} \cap \mathcal{N}_+) \big), \qquad
\alpha = \inf_{z \in \mathbb{D}(z_0, 2r_0)} \Delta Q(z), \\
A =  \sup_{\dist(z, \mathcal{S})\leq 1} |\Delta Q(z)|, \qquad
\beta = \sup_{z \in \mathcal{S}} Q(z),  \\
 l = \inf_{z \in \mathcal{S}} \frac{1}{(1+|z|^2)^2} =  \inf_{z \in \mathcal{S}} \Delta \log (1+|z|^2) . 
\end{gather*}
We have $\alpha > 0$; this is because $Q$ is strictly subharmonic in 
the interior of $\mathcal{S} \cap \mathcal{N}_+$. 
We choose two positive numbers $M_0$ and $M_1$ so that 
\begin{equation} \label{M0andM1}
M_1 \log(1+|z|^2) \leq M_0 \widehat Q(z), \quad z \in \mathbb{C};
\end{equation}
this is possible because of the assumption $\widehat Q \geq 1$ and the growth 
condition 
\[ \widehat Q(z) = \log |z|^2 + \mathrm{O}(1), \quad z \to \infty. \]
Notice that this implies immediately $M_1 \leq M_0$. We will also need 
to assume that $M_1 > 2$. 

We define
\[ 
\phi_m = mQ, \quad \widehat \phi_m(z)= (m-M_0)\widehat Q + M_1 \log(1+|z|^2). 
\] 
The definitions are exactly as in \cite{ahm1}. 
Intuitively, we would like to think $\widehat \phi_m$ to be equal to $m \widehat Q$; 
it however turns out that two correction terms involving the constants 
$M_0$ and $M_1$ are needed. 
The purpose of the logarithmic correction term $M_1 \log(1+|z|^2)$
is 
that now the weight $\widehat \phi_m$ becomes strictly subharmonic in the whole plane: 
\begin{equation} \label{deltahatphi-ineq}
\Delta \widehat \phi_m(z) = (m-M_0)\Delta \widehat Q(z) + M_1(1+|z|^2)^{-2} \geq M_1(1+|z|^2)^{-2},
\end{equation}
for a.e. $z \in \mathbb{C}$ and $m \geq M_0$. 
The effect of substracting the term $M_0 \widehat Q$ is
\begin{equation} \label{hatphi-ineq1}
 \widehat \phi_m \leq \phi_m, \quad z \in \mathbb{C}. 
\end{equation}
We also see that 
\begin{equation} \label{hatphi-ineq2}
\phi_m \leq \widehat \phi_m + M_0 \beta.
\end{equation}
on $\mathcal{S}$. 

We fix $r_1 < r_0$ and suppose that $2m^{-1/2} \leq r_1$. We will work with 
a function $\rho_m$ which is $C^{1,1}$-smooth on $\mathbb{C}$, zero on $\mathbb{D}(z_0, m^{-1/2})$, constant 
on $\mathbb{C} \backslash \mathbb{D}(z_0, r_1-m^{-1/2})$  and 
satisfies the properties
\begin{gather} 
- \frac{m \alpha}{2} + M_0 \alpha \leq \Delta \rho_m \leq mA  \label{rhocond1} \\
|\bar \partial \rho_m|^2 \leq \frac{\alpha}{27\cdot 16 \cdot e^{2A+1+M_0\beta}}\big( \frac12m\alpha+M_1l \big)  \label{rhocond2} \\
|\bar \partial^2 \rho_m + (\bar \partial \rho_m)^2|^2 \leq 
\frac{\big( \frac12 m \alpha+ M_1l \big)\big( \frac14 m \alpha+ M_1l \big)}{9e^{M_0 \beta}} \quad \text{a.e.} \label{rhocond3}
\end{gather}
for $m \geq 1$. The role of these conditions will become clear in the proof of theorem \ref{rhoprop}. 
A consequence of the estimate \eqref{rhocond1} is that 
\begin{equation} \label{laplacephihatrho1}
\Delta (\widehat \phi_m + \rho_m) \geq \frac{m \alpha}{2} + M_1 l
\end{equation}
for a.e. $ z \in \mathbb{D}(z_0, 2r_0)$. We also have
\begin{equation} \label{laplacephihatrho2}
 \Delta (\widehat \phi_m + \rho_m) \geq M_1 \frac{1}{(1+|z|^2)^2} 
\end{equation}
for a.e. $z \in \mathbb{C}$. 

Our argument is based on iterating the elementary one-dimensional 
version of H\"ormander's $\bar \partial$-estimate. 
This result states that for $f \in L^2_{loc}(\mathbb{C})$ and $\psi\in C^{1,1}(\mathbb{C})$ satisfying 
$\Delta \psi > 0 $ a.e., there exists a  solution to the equation 
\begin{equation} \label{horre1}
\bar \partial u = f
\end{equation}
satisfying
\[
\int_{\mathbb{C}}|u|^2 e^{-\psi} \leq \int_{\mathbb{C}} |f|^2\frac{e^{-\psi}}{\Delta \psi}
\]
provided that the right hand side is finite. It is also known 
that there exists a unique norm-minimal solution $u_0$ to \eqref{horre1}. 
Namely, given any solution 
$u_1 \in L^2(e^{-\psi})$ to the equation, $u_0$ can be written as
\[
u_0 = u_1 - \mathrm{P}_{\psi}[u_1],
\]   
where $\mathrm{P_\psi}$ denotes the projection to the subspace of analytic functions 
within $L^2(e^{-\psi})$. 

\begin{thm} \label{bihormander}
Let $f \in L^{\infty}(\mathbb{C})$ be supported on $\mathbb{D}(z_0, r_0)$ and 
let $\rho_m$ satisfy the conditions above.  
Then, there exists a solution $u_2$ to the equation 
\[
\bar \partial^2 u = f
\]
satisfying 
\begin{equation}
\int_{\mathbb{C}} |u_2|^2 e^{-(\widehat \phi_m + \rho_m)} \diff A \\
\leq \frac{1}{\big( \frac12 m \alpha + M_1 l \big) \big( \frac14 m \alpha +M_1l \big)}
\int_{\mathbb{D}(z_0,r_0)} |f|^2 e^{-(\widehat \phi_m + \rho_m)} \diff A
\end{equation}
for all $m \geq m_0$, where $m_0$ is a positive constant depending only on the parameters 
$M_1, l$, $\alpha$ and $r_0$. 
\end{thm}
\begin{proof}
Let us consider the equation 
\[ \bar \partial u = u_1, \]
where $u_1 \in L^2_{loc}(\mathbb{C})$ is some function to be specified in a moment. 

According to H\"ormander's result and \eqref{laplacephihatrho2},  
there exists a solution $u_2$ to the equation with the norm control 
\begin{multline} \label{bihorre-eq1}
\int_{\mathbb{C}} |u_2|^2 e^{-(\widehat \phi_m + \rho_m)} \diff A \leq 
\int_{\mathbb{C}} |u_1|^2 \frac{e^{-(\widehat \phi_m + \rho_m)}}{\Delta (\widehat \phi_m + \rho_m)} \diff A \\
\leq \int_{\mathbb{C}} |u_1|^2 e^{-[(\widehat \phi_m + \rho_m) + \log \Delta (\widehat \phi_m + \rho_m)]} \diff A.
\end{multline}
We would like to proceed by using H\"ormander's estimate again with the weight 
$
(\widehat \phi_m + \rho_m) + \log \Delta (\widehat \phi_m + \rho_m)
$. But
this is not possible, because the function $\log \Delta (\widehat \phi_m + \rho_m)]$ is not $C^{1,1}$-smooth. 
We proceed by replacing it with another function which is smooth enough. 
Building on \eqref{laplacephihatrho1} and \eqref{laplacephihatrho2}, we estimate
\[
\log \Delta (\widehat \phi_m + \rho_m) \geq \theta_m(|z-z_0|),  
\]
where $\theta_m$ is any $C^{1,1}$-smooth function satisfying
\[
\theta_m(x) \leq \log \big( \frac12 \alpha m + M_1 l \big), \qquad 0 < x \leq 2r_0,
\] 
and 
\[
\theta_m(x) \leq \log M_1 - 2 \log \big(1+|z|^2 \big)
\]
for $x \geq 2r_0$. 
It will be convenient to make the change of variables $x= e^t$, so that we can rewrite
\[ 
\Delta \theta_m(|z-z_0|) = \frac{1}{4|z-z_0|^2}\frac{\diff^2 }{\diff t^2} \theta_m(e^t)_{\mid t= \log |z-z_0|}.
\]
We define another function by $\sigma_m(t):= \theta_m(e^t)$, and set 
\[\sigma_m(t) = \log \bigg[ \frac12 m \alpha+M_1l \bigg],  \quad -\infty < t \leq \log r_0, \]
and 
\[ 
\sigma_m(t) = \log M_1 -2 \log(1+ e^{2t}),  \quad t \geq \log 2 + \log r_0.
\]
It remains to define $\sigma_m$ on the interval $[\log r_0, \log r_0 + \log 2]$. We set 
\[
\sigma_m(t) := \log \bigg[ \frac12 m \alpha+M_1l \bigg]- \frac{m \alpha r_0^2}{2} (t-\log r_0)^2, 
\qquad \log r_0 \leq t \leq t_m,
\]
where 
\[
t_m := \log r_0 +\frac{\sqrt{2}}{\sqrt{m \alpha}r_0}
\sqrt{\log \bigg[ \frac{m\alpha}{2M_1}+l \bigg]}.
\]
The motivation for the definition of  $t_m$ is that 
it solves the equation $\sigma(t_m)= \log M_1$.
We are here assuming that $m \geq m_0$, where $m_0$ 
is so large that 
\[
\log \bigg[ \frac{m\alpha}{2M_1}+l \bigg] > 0 
\]
and  
$t_m \leq \log r_0 + \frac12 \log 2$ for all such $m$.  Notice that this definition
of $\sigma_m$ implies that on the annulus 
$\mathbb{D}(z_0, t_m) \backslash \mathbb{D}(z_0, r_0)$ we have
\[
\Delta \theta_m(|z-z_0|) = -\frac{m\alpha r_0^2}{4|z-z_0|^2} \geq -\frac{m \alpha}{4}.
\]

For the remaining interval $[t_m, \log2 + \log r_0]$, the choice of $\sigma_m$ 
is rather insignifigant. We can namely choose any decreasing function so that 
the resulting function 
defined on the whole real line becomes     
$C^{1,1}$  and whose second derivative 
is not smaller than $-\alpha r_0^2m$. 

With this construction of $\sigma_m$, we now see 
from \eqref{laplacephihatrho1} and 
\eqref{laplacephihatrho2} that
\[
\Delta \big( \widehat \phi_m(z)+ \rho_m(z) + \theta_m(|z-z_0|)\big) 
\geq \frac14 m \alpha + M_1l,
\]
for $z \in \mathbb{D}(z_0, 2r_0)$ and
\[
\Delta \big( \widehat \phi_m(z)+ \rho_m(z) + \theta_m(|z-z_0|)\big) 
\geq M_1 \big(1+|z|^2)^{-2} -2\big(1+|z|^2)^{-2}
\]
for $z \in \mathbb{C} \backslash \mathbb{D}(z_0,2r_0)$. 
In any case, $\widehat \phi_m + \rho_m + \theta_m(|z-z_0|)$ 
is strictly subharmonic a.e.  in the whole complex plane (we use here the assumption $M_1 > 2$ 
from the beginning of this section). 

We can now apply H\"ormander's estimate the second time. We take
$u_1$ to be thel solution of 
\[
\bar \partial u = f
\]
which is norm-minimal in $L^2(e^{-(\widehat \phi_m + \rho_m + \theta_m(|z-z_0|))})$. 
We can now continue from \eqref{bihorre-eq1}:
\begin{multline}
\int_{\mathbb{C}} |u_1(z)|^2 e^{-( \widehat \phi_m + \rho_m + \log \Delta (\widehat \phi_m + \rho_m))} \diff A(z) 
\leq \int_{\mathbb{C}} |u_1(z)|^2 e^{-(\widehat \phi_m + \rho_m + \theta_m(|z-z_0|))} \diff A(z)\\
\leq \int_{\mathbb{D}(z_0,r_0)} |f(z)|^2 \frac{e^{-(\widehat \phi_m(z) + \rho_m(z) + \theta_m(|z-z_0|))}}
{\Delta \big( \widehat \phi_m(z)+ \rho_m(z) + \theta_m(|z-z_0|)\big)} \diff A(z) \\
\leq \frac{1}{\big( \frac12 m \alpha + M_1 l \big) \big( \frac14 m \alpha +M_1l \big)}
\int_{\mathbb{D}(z_0,r_0)} |f(z)|^2 e^{-(\widehat \phi_m(z) + \rho_m(z))} \diff A(z),
\end{multline}
and the proof is complete.
\end{proof}

We will now prove a generalization of lemma \ref{pointwiselemma}. 
\begin{lem} \label{pointwiselemma2}
Take $m > 0$, $z \in \mathbb{C}$ and let $\psi_m$ be a $C^{1,1}$-smooth real-valued function on 
$\mathbb{C}$. 
Then, for any bianalytic function $u$, 
we have 
\begin{equation} \label{pointwise21}
|u(z)|^2e^{-\psi_m(z)} \leq m(8+48A^2)e^{A_m} \int_{\mathbb{D}(z, m^{-1/2})} |u(w)|^2 e^{-\psi_m(w)} \diff A(w)
\end{equation}
and 
\begin{equation} \label{pointwise22}
|\bar \partial u(z)|^2e^{-\psi_m(z)} \leq 3m^2e^{A_m} \int_{\mathbb{D}(z, m^{-1/2})} |u(w)|^2 e^{-\psi_m(w)} \diff A(w),
\end{equation}
where 
\[
A_m:= \frac{1}{m} \esssup_{w \in \mathbb{D}(z, m^{-1/2})} \big| \Delta \psi_m(w) \big|
\]
\end{lem}
\begin{proof}
We assume $z=0$ without loss of generality. 
In \cite{hh2}, it was proved that for a bianalytic function $v$ on $\mathbb{D}$ 
and a subharmonic function $\Psi \in C^{1,1}(\mathbb{D})$ satisfying 
\[
\int_{\mathbb{D}} (1-|w|^2) \Delta \Psi(w) \diff A(w),
\]
it holds 
\begin{equation} \label{subharmonic1} 
|v(0)|^2e^{\Psi(0)} \leq \bigg[8 + 12 |\mathbf{G}[\Delta \Psi](0)|^2 \bigg] \int_{\mathbb{D}} |v(w)|^2 e^{\Psi(w)} \diff A(w), 
\end{equation}
where $\mathbf{G}[\mu]$ denotes the Green's potential of the measure $\mu$: 
\[ 
\mathbf{G}[\mu](z):= \int_{\mathbb{D}} \log \bigg| \frac{z-w}{1- z \bar w} \bigg|^2 \diff \mu(w).
\]
We will set $u_m(\xi)= u(m^{-1/2}\xi)$ and 
$\Psi_m(\xi) = A_m|\xi|^2 -  \psi_m(m^{-1/2} \xi)$ for $\xi \in \mathbb{D}$. We have  $\Delta \Psi_m \geq 0$ 
and 
\[
\big|\mathbf{G}[\Delta \Psi_m](0)\big| \leq 2A_m.
\]
An application of \eqref{subharmonic1} gives 
\begin{multline} 
|u(0)|^2e^{-\psi_m(0)} = |u_m(0)|^2e^{\Psi_m(0)} \leq 
\bigg[8 + 12 \big|\mathbf{G}[\Delta \Psi_m](0) \big|^2 \bigg] 
\int_{\mathbb{D}} |u_m|^2 e^{\Psi_m} \diff A 
\\ 
\leq (8 + 48 A_m^2) e^{A_m} 
\int_{\mathbb{D}} |u_m(\xi)|^2 e^{-\psi_m(m^{-1/2}\xi)} \diff A(\xi)\\
= m (8+48A_m^2)e^{A_m} \int_{\mathbb{D}(0,m^{-1/2})} |u(w)|^2 e^{-\psi_m(w)} \diff A(w),
\end{multline}
which proves the first inequality. The second inequality follows the same way from  
\begin{equation} \label{subharmonic2} 
|\bar \partial v(0)|^2e^{\Psi(0)} \leq 3 \int_{\mathbb{D}} |v(w)|^2 e^{\Psi(w)} \diff A(w),
\end{equation}
which also was proved in \cite{hh2}. 
\end{proof}

In the following lemma and later, we will use the notation 
$]x[$ for the largest integer smaller than or equal to a real number $x$.
\begin{lem} \label{bianalpol}
Let $f$ and $g$ be analytic functions, and suppose that 
\[
u(z):= f(z)+\bar z g(z) \in L^2\big( e^{-\widehat \phi_m} \big), \qquad m \geq M_0. \]
Then,  
$f$ and $g$ are polynomials of degree $\leq ]m-M_0+M_1+1[$.
\end{lem}
\begin{proof}
From lemma \ref{pointwiselemma2}, we get
\[
|g(z)|^2 e^{-\widehat \phi_m(z)} \leq 3m^2e^{A_{z,m}} \int_{\mathbb{D}(z, m^{-1/2})}
|u(w)|e^{-\widehat \phi_m(w)} \diff A(w),
\]
where 
\[
A_{z,m} := \frac1m \esssup_{w \in \mathbb{D}(z, m^{-1/2})} \big| \Delta \widehat \phi_m(w)|.
\]
Using \eqref{deltahatphi-ineq}, we see that 
\begin{equation} \label{Azm}
A_{z,m} \leq K+1, \qquad z \in \mathbb{C}, \quad m \geq M_0,
\end{equation}
where 
\[
K: = \esssup_{z \in \mathcal{S}} \Delta \widehat Q(z) = \esssup_{z \in \mathbb{C}} \Delta \widehat Q(z).
\]
The last equality followed from harmonicity of $\widehat Q$ in $\mathbb{C} \backslash \mathcal{S}$. 
By \eqref{Azm}, we have 
\[ |g(z)|^2 \leq 3m^2e^{K+1} \|u\|_{L^2(e^{-\widehat \phi_m})}^2e^{\widehat \phi_m(z)}. \]
for all $z \in \mathbb{C}$. 
This, together with the growth condition
\begin{equation} \label{phihatgrowth}
\widehat \phi_m(z) = (m-M_0+M_1) \log|z|^2 + \mathrm{O}(1) 
\end{equation}
shows that $g$ is a polynomial of degree $\leq  ]m-M_0+M_1[$. 
For $f$, we first estimate
\[ |f(z)|^2
\leq 2(|u(z)|^2+ |zg(z)|^2). \]
We use \eqref{pointwise21} with $\psi_m= \widehat \phi_m$ for the first term:
\[
|u(z)|^2 \leq m[8+48(K+1)^2]e^{K+1} \|u\|^2_{L^2(e^{-\widehat \phi_m})} \e^{\widehat \phi_m(z)}.
\] 
We get 
\begin{multline}
|f(z)|^2 \leq 2(|u(z)|^2+ |zg(z)|^2) \\ 
\leq 
2m[8+48(K+1)^2]e^{K+1} \|u\|^2_{L^2(e^{-\widehat \phi_m})} \e^{\widehat \phi_m(z)}
+ 2|zg(z)|^2.
\end{multline}
The desired assertion for $f$ now follows from the growth condition \eqref{phihatgrowth} and 
that $zg(z)$ is a polynomial of degree $\leq  ]m-M_0+M_1+1[$
\end{proof}

We now turn to a definition of certain 
subspaces of $L^2(e^{-\psi})$, where $\psi$ is some real-valued $C^{1,1}$-smooth 
weight function on 
$\mathbb{C}$. We assume that $\int_{\mathbb{C}} e^{-\psi} \diff A < \infty$ so that 
all constant functions belong to the space. Let us first fix some notation: 
\[
A^2_{2,\psi} := \{ f \mid f \in L^2(e^{-\psi}), \bar \partial^2 f=0 \}
\]
and 
\[
A^2_{2,\psi, n} := L^2(e^{-\psi}) \cap \mathrm{Pol}_{2,n}.
\]
The spaces in question are
\begin{multline}
 L^2_{2, \psi, n} := \{ f \in L^2(e^{-\psi})  
\mid f \in C^1(\mathbb{C} \backslash K) \text{ for some compact set } K, \\
|f(z)-\bar z \bar \partial f(z)| =\mathrm{O}(|z|^{n-1}), \bar \partial f(z) = \mathrm{O}(|z|^{n-1}) 
\text{ as } z \to \infty \}. 
\end{multline}
We see immediately that these spaces are not in general closed in
$L^2(e^{-\psi})$. 
An important fact however is that 
\begin{equation} \label{growth-charac}
A^2_{2, \psi, n} = L^2_{2, \psi, n} \cap \mathrm{A}^2_{2, \psi},
\end{equation}
and $A^2_{2,\psi, n}$ is always closed in $L^2(e^{-\psi})$. 
This makes it possible to speak about $L^2_{2,\psi, n}$-minimal solutions 
to $\bar \partial^2$-equations. 

We will make use of the fundamental solution $\frac{\bar z}{z}$ of the operator $\bar \partial^2$ 
and denote the associated solution operator by $\mathrm{C}_2$:
\[ \mathrm{C}_2[f](z) = \int_{\mathbb{C}} f(w) \frac{\overline{z-w}}{z-w} \diff A(w) \]
for a compactly supported $L^{\infty}$-function $f$. 
We also write $\mathrm{P}_{2, \psi, n}$
for the orthogonal projection from $L^2(e^{-\psi})$ to $A^2_{2,\psi,n}$.

Because $\mathrm{C}_2[f] \in L^2_{2,\psi, n}$, we have by \eqref{growth-charac} that 
\[
u_n := \mathrm{C}_2[f]- \mathrm{P}_{2, \psi, n}\big[ \mathrm{C}_2[f] \big]
\]
is the unique $L^2_{2,\psi,n}$-minimal solution to the equation 
\[
\bar \partial^2 u = f.
\]
This solution is characterized among the solutions in the class $L^2_{2,\psi,n}$ by the condition
\[
\int_{\mathbb{C}} u_n(z) \overline{h(z)} e^{-\psi} \diff A(z) = 0 
\]
which should hold for all $h \in A^2_{2,\psi,n}$. 

The following lemma is an adaption of lemma 4.2 in \cite{ahm1} to our bianalytic context. 
\begin{lem} \label{normminsollemma}
For any compactly supported function $f \in L^{\infty}(\mathbb{C})$, the equation 
\[ \bar \partial^2 u = f \]
has a solution in $L^2\big(e^{-(\widehat \phi_m + \rho_m)}\big)$. 
The $L^2\big(e^{-(\widehat \phi_m+\rho_m)}\big)$-minimal solution is 
of class $L^2_{2,\phi_m, n}$ when $n \geq ]m-M_0+M_1+1[$. 
\end{lem}
\begin{proof}
The function $\mathrm{C}_2[f]$ solves the equation and satisfies $\mathrm{C}_2[f] = \mathrm{O}(1)$ 
as $z \to \infty$. Therefore, it belongs to the class $L^2\big(e^{-(\widehat \phi_m+\rho_m)} \big)$. 
The $L^2\big(e^{-(\widehat{\phi}+\rho_m)}\big)$-minimal solution $v$ can be written as
\[
v= C_2[f] + g
\]
for some bianalytic function 
$g \in A^2_{2,\widehat \phi_m + \rho_m}$.  
Because we assume that $\rho_m$ is $C^{1,1}$-smooth and constant 
in $\mathbb{D}(z_0, r_1-m^{-1/2})$, 
the spaces $A^2_{2, \widehat \phi_m + \rho_m}$ and 
$A^2_{2, \widehat \phi_m}$ equal as sets, and therefore, by \eqref{hatphi-ineq1} and lemma \ref{bianalpol}, 
we get the inclusion 
\[
A^2_{2, \widehat \phi_m + \rho_m} \subset A^2_{2, \phi_m, n}.
\]
Hence, $g \in A^2_{2, \phi_m, n}$, and also $g \in L^2_{2,\phi_m, n}$. 
Because $C_2[f]$ is bounded and $\bar \partial \mathrm{C_2}[f] 
= \mathrm{O}(|z|^{-1})$ as $z \to \infty$, we see that  $\mathrm{C}_2[f] \in L^2_{2,\phi_m, n}$. 
The assertion is now proved.  
\end{proof}

\begin{thm} \label{rhoprop}
Fix a number $r_1$ such that $r_1 \leq r_0$ and let $f \in L^{\infty}(\mathbb{C})$
be supported on $\mathbb{D}(z_0, r_0)$ and equal to $0$ 
on $\mathbb{D}(z_0, r_1)$. Assume
\[
 m \geq \max\{m_0, M_0\}, \quad 2m^{-1/2} < r_1, \quad n \geq ]m-M_0+M_1+1[,
\]
where $m_0$ is as in proposition \ref{bihormander}. Let $\rho_m$ be 
a $C^{1,1}$-smooth function which is supported on $\mathbb{D}(z_0, r_1-m^{-1/2})$, 
zero on $\mathbb{D}(z_0, m^{-1/2})$ 
and satisfies the conditions \eqref{rhocond1}, \eqref{rhocond2} and \eqref{rhocond3}. 
Then, the $L^2_{2,mQ,n}$-minimal solution 
$u_n$ to the equation $\bar \partial^2 u = f$ 
satisfies 
\begin{equation}
\int_{\mathbb{C}} |u_n|^2 e^{\rho_m-mQ} 
\leq \frac{9e^{M_0\beta}}
{\big(\frac14 m \alpha + M_1l \big)\big( \frac12 m \alpha + M_1l \big)}
 \int_{\mathbb{D}(z_0,r_0)} |f|^2 e^{\rho_m - mQ}. 
\end{equation}
In the special case $\rho_m=0$, the conditions $2m^{-1/2} < r_1$ and 
$f=0$ on $\mathbb{D}(z_0,r_1)$ are not needed.
\end{thm}
\begin{proof}

The norm-minimality
condition for $u_n$ can be rephrased as
\begin{equation} \label{normminimality}
\int_{\mathbb{C}} u_n e^{ \rho_m} \bar h e^{-(\phi_m +\rho_m)} \diff A = 0 \quad \text{for all } h \in A^2_{2, \phi_m+ \rho_m, n} .
\end{equation}
This is because $A^2_{2, \phi_m, n}$ and $A^2_{2,\phi_m+\rho_m, n}$ are equal as sets. Because $\rho_m$ is supported 
on a compact set, we have $u_ne^{\rho_m} \in L^2_{2,\phi_m+\rho_m,n}$. Therefore, $u_ne^{\rho}$ is the 
$L^2_{2,\phi_m+\rho_m, n}$-minimal solution to the equation
\begin{equation} \label{twistedprob}
\bar \partial^2 u = \bar \partial^2( u_n e^{\rho_m}) =
\big[ u_n (\bar \partial^2 \rho_m + (\bar \partial \rho_m)^2) +2 \bar\partial u_n \bar \partial \rho_m 
+ f \big] e^{\rho_m}.
\end{equation}

Let $u_2$ be the $L^2_{\widehat \phi_m+ \rho_m}$-minimal solution of the problem \eqref{twistedprob}, the existence 
of which is guaranteed by lemma \ref{normminsollemma}.  
We deduce from the same lemma that $u_2 \in L^2_{2, \phi_m, n}$, and consequently 
$u_2 \in L^2_{2, \phi_m+\rho_m, n}$. Hence 
\begin{gather} \label{rhopropeq1}
\int_{\mathbb{C}} |u_n|^2 e^{\rho_m - \phi_m} \diff A = 
\int_{\mathbb{C}} |u_n e^{\rho_m}|^2 e^{-(\phi_m + \rho_m)} \diff A
\leq \int_{\mathbb{C}} |u_2|^2 e^{-(\phi_m + \rho_m)} \diff A \\
\leq \int_{\mathbb{C}} |u_2|^2 e^{-(\widehat \phi_m + \rho_m)} \diff A.
\end{gather}
For the last inequality here, we applied \eqref{hatphi-ineq1}. 
The proposition \ref{bihormander} gives 
\begin{multline} \label{threeterms}
\int_{\mathbb{C}} |u_2|^2 
e^{-(\widehat \phi_m + \rho_m)} \diff A 
\leq \frac{1}{\big( \frac12 m\alpha + M_1l\big)\big( \frac14 m \alpha+M_1l \big)} 
\int_{\mathbb{C}}
| \bar \partial^2 ( u_n e^{\rho_m})|^2  
e^{-(\widehat \phi + \rho_m)} \diff A \\
\leq  \frac{1}{\big( \frac12 m\alpha + M_1l\big)\big( \frac14 m \alpha+M_1l \big)} 
\bigg[ 3\int_{\mathbb{D}(z_0, r_1-m^{-1/2})} |u_n|^2
\big| \bar \partial^2 \rho_m + (\bar \partial \rho_m)^2 \big|^2 e^{\rho_m - \widehat \phi} \diff A \\
+  12 \int_{\mathbb{D}(z_0, r_1-m^{-1/2})} | \bar \partial u_n|^2 
|\bar \partial \rho_m |^2
e^{\rho_m - \widehat \phi} \diff A 
+ 3\int_{\mathbb{D}(z_0, r_0)} |f|^2 e^{\rho_m - \widehat \phi} \diff A \bigg].
\end{multline}
We proceed to estimate the three terms in this expression. For the first term, 
we use \eqref{rhocond3}:
\begin{multline} \label{rhopropfirstterm}
\frac{3}{\big( \frac12 m\alpha + M_1l\big)\big( \frac14 m \alpha+M_1l \big)} 
 \int_{\mathbb{D}(z_0, r_1-m^{-1/2})} |u_n|^2 |\bar \partial^2 \rho_m + (\bar \partial \rho_m)^2|^2
 e^{\rho_m - \widehat \phi_m} \diff A 
\\ \leq \frac{1}{3 e^{M_0\beta}}
 \int_{\mathbb{D}(z_0, r_1-m^{-1/2})} | u_n|^2
 e^{\rho_m - \widehat \phi_m} \diff A \\
\leq \frac13 \int_{\mathbb{D}(z_0, r_1-m^{-1/2})} |u_n|^2 e^{\rho_m- \phi_m} \diff A.
\end{multline}

Let us turn to the second term in \eqref{threeterms}.  
The condition $f=0$ on $\mathbb{D}(z_0, r_1)$ means that 
$u_n$ is bianalytic in this disk. We will employ the estimate 
\eqref{pointwise22} for this function, with the weight
$\psi_m = \widehat \phi_m - \rho_m$. We first check that 
\[
\frac1m \sup_{w \in \mathbb{D}(z_0, r_1)} |\Delta \big(\widehat \phi_m(w) - \rho_m(w)\big)|
\leq \frac1m [(m-M_0)A + M_1 + mA]\leq 2A+1,
\]
where we used $|\Delta \rho_m|\leq mA$, which is a direct consequence of 
\eqref{rhocond1}. We get with \eqref{rhocond2} 
\begin{multline} \label{rhopropsecondterm}
\frac{12}{\big( \frac12 m\alpha + M_1l\big)\big( \frac14 m \alpha+M_1l \big)} 
\int_{\mathbb{D}(z_0, r_1-m^{-1/2})} | \bar \partial u_n|^2 |\bar \partial \rho_m |^2 e^{\rho_m- \widehat \phi_m} \diff A \\
\leq \frac{\alpha}{36e^{2A+1}\big( \frac14m\alpha+ M_1l \big)e^{M_0 \beta}} 
\int_{\mathbb{D}(z_0, r_1-m^{-1/2})} | \bar \partial u_n|^2 e^{\rho_m- \widehat \phi_m} \diff A \\
\leq  \frac{\alpha m^2}{12\big( \frac14m\alpha+ M_1l \big)e^{M_0 \beta}} 
 \int_{\mathbb{D}(z_0, r_1-m^{-1/2})} \int_{|w-z| \leq m^{-1/2}} |u_n(w)|^2 e^{\rho_m(w)- \widehat \phi_m(w)} \diff A(w) \diff A(z)
\\
\leq   \frac{\alpha m^2}{12\big( \frac14m\alpha+ M_1l \big)e^{M_0 \beta}} \int_{\mathbb{D}(z_0, r_1)} |u_n(w)|^2 e^{\rho_m(w)- \widehat \phi_m(w)} 
\int_{|w-z| \leq m^{-1/2}} \diff A(z) \diff A(w) \\
= \frac{1}{3 e^{M_0\beta}}  \int_{\mathbb{D}(z_0, r_1)} |u_n(w)|^2 e^{\rho_m(w)- \widehat \phi_m(w)} \diff A(w)  
\leq \frac13  \int_{\mathbb{D}(z_0,r_1)} |u_n(w)|^2 e^{\rho_m(w)- \phi_m(w)} \diff A(w).
\end{multline}

The third term, we only want to replace $\widehat \phi_m$ by $\phi_m$:  
\begin{multline} \label{rhopropthirdterm}
\frac{3}{\big( \frac12 m\alpha + M_1l\big)\big( \frac14 m \alpha+M_1l \big)}
\int_{\mathbb{D}(z_0, r_0)} |f|^2  e^{\rho_m - \widehat \phi_m} \diff A
\\
\leq 
 \frac{3e^{M_0\beta}}{\big( \frac12 m\alpha + M_1l\big)\big( \frac14 m \alpha+M_1l \big)}
\int_{\mathbb{D}(z_0, r_0)} |f|^2 e^{\rho_m - \widehat \phi_m} \diff A.
\end{multline}
Now, combining \eqref{rhopropeq1}, \eqref{threeterms}, \eqref{rhopropfirstterm}, \eqref{rhopropsecondterm} and \eqref{rhopropthirdterm} gives, 
\begin{multline}
\int_{\mathbb{C}} |u_n|^2 e^{\rho_m -  \phi_m} \diff A \\
\leq 
\frac23 \int_{\mathbb{C}} |u_n|^2 e^{\rho_m - \phi_m} \diff A 
+ 
\frac{3e^{M_0\beta}}{\big( \frac12 m\alpha + M_1l\big)\big( \frac14 m \alpha+M_1l \big)}
 \int_{\mathbb{D}(z_0, r_1)} |f|^2 e^{\rho_m - \phi_m} \diff A, 
\end{multline}
from which the assertion of the proposition follows immediately. 

Notice that the condition $f=0$ on $\mathbb{D}(z_0,r_1)$ was only used in \eqref{rhopropsecondterm}. 
If $\rho_m=0$, this assumption is clearly not needed. 
For the same reason, it is unnecessary to require that $2m^{-1/2}\leq r_1$ . 
\end{proof}

We will make use of this theorem twice, in the proofs of both main theorems \ref{blowupthm} and \ref{offdiagonalthm}. 
For the first one, it is enough to use the special case $\rho_m=0$. 

\section{Local blow-up for bianalytic Bergman kernels in the bulk} \label{blowupsection}
We are now ready to present the proof the theorem \ref{blowupthm}. The main ingredients 
of the argument are the construction of local bianalytic Bergman kernels in section 
\ref{bianalyticsection} and the estimate for $\bar \partial^2$-operator in 
theorem \ref{rhoprop}. The structure of the proof is the same as in \cite{hh2}; the main difference 
is that there it was enough to apply a simpler $\bar \partial^2$-estimate. 

\begin{proof} [Proof of theorem \ref{blowupthm}]
We pick a point $z_0 \in \mathcal{S} \cap \mathcal{N}_+$ and let $\chi$ and $r$ 
be as in section \ref{localkernelsection}. We require addionally that
\[
r \leq r_0 = \frac14 \dist(z_0, \mathbb{C} 
\backslash \mathcal{S} \cap \mathcal{N}_+). 
\] 

Recall that we constructed a local bianalytic Bergman kernel 
$\mathrm{mod}(m^{-\frac12})$ in the form  
\begin{multline} \label{bilocalkerneltwoterms}
K^{\langle 2 \rangle}_{2,mQ}(z,w)  
:= \Lfun^{\langle 2 \rangle}_{2,mQ}(z,w) e^{mQ(z,w)} \\
= \big(m^2 \Lfun_{2,0}(z,w) + m\Lfun_{2,1}(z,w) \big) e^{mQ(z,w)}, \quad z, w \in \mathbb{D}(z_0, r),
\end{multline}
where $\Lfun_{2,0}$ and $\Lfun_{2,1}$ are as in \eqref{b22}, \eqref{b21}. 

Let us define the associated integral operator as 
\[ I^{\langle 2 \rangle}_{2,mQ} [f](\zeta) = \int_{\Omega} f(\xi) K^{\langle 2 \rangle}_{2,mQ}(\zeta, \xi) 
\chi(\xi) e^{-mQ(\xi)} \diff A(\xi), \quad \zeta \in \mathbb{D}(z_0, r). \]
We write $K_{2,mQ,n;z}(w) := K_{2,mQ,n}(w,z)$, with a similar notation for kernels 
$K^{\langle 2 \rangle}_{2,mQ}$. It is easily verified that 
\[ \mathrm{P}_{2,mQ,n}[K^{\langle 2 \rangle}_{2,mQ; w} \chi](z) = \overline{I^{\langle 2 \rangle}_{2,mQ}[K_{2,mQ,n;z}](w) }, \qquad z,w \in \mathbb{D}(z_0,r), \]
where $\mathrm{P}_{2,mQ,n}$ denotes the projection within $L^2(e^{-mQ})$ to the subspace
$A^2_{2,mQ,n}$. 

Our first step is to make the splitting 
\begin{multline} 
 \bigg| K_{2,mQ,n}(z,w) - K^{\langle 2 \rangle}_{2,mQ}(z,w) \bigg| 
\leq \bigg|  K_{2,mQ,n}(w, z) - I^{\langle 2 \rangle}_{q,m}[K_{2,mQ;z}](w) \bigg| \\
+ \bigg| \mathrm{P}_{2,mQ,n}[K^{\langle 2 \rangle}_{2,mQ; w} \chi](z) - K^{\langle 2 \rangle}_{2,m ;w}(z) \chi(z) \bigg|, 
\qquad z,w \in \mathbb{D}(z_0,\frac13 r).
\label{splitting}
\end{multline}
For the first term on the right hand side, we use that $K^{\langle 2 \rangle}_{2,mQ}$ is a local bianalytic 
Bergman kernel $\mathrm{mod}(m^{-\frac12})$. 
This fact can be rephrased as
\begin{multline} \label{splittingterm1}
\bigg| K_{2,mQ}(w, z) - I^{\langle 2 \rangle}_{2,mQ}[K_{2,mQ;z}](w) \bigg| \leq C_1 m^{-\frac12} e^{\frac12 mQ(w)}
\| K_{2,mQ; z} \|_m \\
= C_1 m^{-\frac12} e^{\frac12 mQ(w)}K_{2,mQ}(z,z)^{1/2} 
\leq C_2e^{\frac12 m[Q(w)+Q(z)]}, \quad z, w \in \mathbb{D}(z_0, \frac13 r)
\end{multline}
where the last inequality is a consequence of \eqref{diagonalestimate1}. Here, and later in the proof, 
$C_j$ with $j \geq 1$ refers to a constant  that can only depend on $z_0, Q$ and $r$. 

We turn to the second term in \eqref{splitting}. By \eqref{growth-charac}, the function
\[
u_0(z):=K^{\langle 2 \rangle}_{2,mQ ;w}(z) \chi(z) - \mathrm{P}_{2,mQ,n}[K^{\langle 2 \rangle}_{2,mQ,n; w} \chi](z) 
\]
is the $L^2_{2,mQ,n}$-minimal solution to the equation
\begin{multline}
\bar \partial^2 u = \bar \partial^2 \big[K^{\langle 2 \rangle}_{2,mQ; w} \chi](\xi) - K^{\langle 2 \rangle}_{2,m ;w}(z) \chi(\xi) \big]
= 2\bar \partial_w^2 K^{\langle 2 \rangle}_{2,mQ; w} \bar \partial \chi(z) + K^{\langle 2 \rangle}_{2,mQ; w} \bar \partial^2 \chi(z).
\end{multline}
The support of this function is contained in $\mathbb{D}(z_0,r_0)$, 
so we can apply theorem \ref{rhoprop} in the special case $\rho_m=0$. First, we record the estimate 
\[
\big| 2\bar \partial K^{\langle 2 \rangle}_{2,mQ; w} \bar \partial \chi(z) + K^{\langle 2 \rangle}_{2,mQ; w} \bar \partial^2 \chi(z)\big|^2 
\leq C_3m^6e^{2m\mathrm{Re}Q(z,w)}, \quad z,w \in \mathbb{D}(z_0, r), m \geq 1,
\]
which follows directly from the definition of $K^{\langle 2 \rangle}_{2,mQ}$. Here $C_3$ is a constant depending only 
on $z_0, Q$ and $r$. 

Assuming that $m \geq \max\{m_0, M_0 \}$ and 
$n \geq ]m-M_0+M_1+1[$, we obtain 
\begin{multline} \label{normminsolestimate1}
\|u_0\|_m^2 \\
\leq \frac{72e^{M_0\beta}}{m^2\alpha^2} 
\int_{\frac23r \leq |z-z_0| \leq r} 
\big| 
2\bar \partial K^{\langle 2 \rangle}_{2,mQ; w}(z) \bar \partial \chi(z) 
+ K^{\langle 2 \rangle}_{2,mQ; w} \bar \partial^2 \chi(z) \big|^2 e^{-mQ(z)}\diff A(z) \\
\leq C_4m^4 e^{mQ(w)} e^{-\delta m}.
\end{multline}
We now make the 
assumption that $m^{-1/2} \leq \frac13 r$, and apply \eqref{pointwise1}: 
\begin{equation} \label{normminsolpointwise}
|u_0(z)|^2 \leq C_5m \|u_0\|_m^2e^{mQ(z)}, \quad z \in \mathbb{D}(z_0, \frac13 r).
\end{equation}
Notice that the use of \eqref{pointwise1} was justified given our extra assumption for $m$, 
because $u_0$ is bianalytic for all $z \in \mathbb{D}(z_0, \frac23 r)$. 

We have by \eqref{normminsolestimate1}  and \eqref{normminsolpointwise} that
\begin{equation} \label{splittingterm2}
|u_0(z)|^2 \leq C_6 m^5 e^{mQ(w)+mQ(z)} e^{-\delta m}
\leq C_7 e^{mQ(w)+mQ(z)}, \qquad z,w \in \mathbb{D}(z_0, \frac13 r).
\end{equation}
Putting \eqref{splittingterm1} and \eqref{splittingterm2} together, we get
\begin{equation} \label{bilocalkernelappro}
\bigg| K_{2,mQ,n}(z,w) - K^{\langle 2 \rangle}_{2,mQ}(z,w) \bigg| \leq C_8 e^{\frac12m(Q(z)+Q(w))}.
\end{equation}
This implies that 
\begin{equation} \label{bilocalkernelappro1}
\big|K_{2,mQ,n}(z,w)\big|e^{-\frac12m(Q(z)+Q(w))}
= \big|K^{\langle 2 \rangle}_{2,mQ}(z,w) \big|e^{-\frac12m(Q(z)+Q(w))} + \mathrm{O}(1)
\end{equation}

Now, we take 
\[
z=z_0+ \frac{\xi}{\sqrt{m \Delta Q(z_0)}}, \quad w = z_0 + \frac{\lambda}{\sqrt{m \Delta Q(z_0)}},
\]
where $(\xi, \lambda)$ is restricted to be in some compact subset of $\mathbb{C}^2$ and $m$ 
is assumed to be so big that $z,w \in \mathbb{D}(z_0, \frac13 r)$. 

Using the explicit form for $\Lfun^{\langle 2 \rangle}_{2,mQ}(z,w)$, we compute
\begin{multline} \label{bilocalkernelblowup1}
 \Lfun^{\langle 2 \rangle}_{2,mQ}(z,w) = \bigg[-m \frac{b(z,w)^2}{[\Delta Q(z_0)]}|\xi-\lambda|^2 
+ m2b(z,w) + \mathrm{O}(m^{1/2})\bigg] \\ 
= m\Delta Q(z_0) (-|\xi-\lambda|^2  +2) + \mathrm{O}(m^{1/2}),
\end{multline}
where for the last equality, the Taylor expansion
\[
b(z, w) = \Delta Q(z_0) + \mathrm{O}(m^{-1/2}).
\]
was applied. 
Furthermore, by \eqref{Qtaylor1}, 
\begin{equation} \label{blowupexponential1}
e^{m\mathrm{Re}Q(z,w)-\frac12m\big(Q(z)+Q(w)\big)}
= e^{-\frac12 |\xi-\lambda|^2} + \mathrm{O}\big(m^{-\frac12}\big).
\end{equation}

Finally, recall that $L^1_1(x)= -x+2$. The theorem now follows from \eqref{bilocalkernelappro1}, 
\eqref{bilocalkernelblowup1} and \eqref{blowupexponential1}.
\end{proof}
\begin{rem}
The formula \eqref{bilocalkernelappro}  provides justification 
for the claim that local bianalytic kernels 
approximate the kernel $K_{2,mQ,n}$ in the bulk (in the presence of
appropriate weight factors, of course).  
By using a local kernel with more terms, one would obtain a better 
approximation. 
\end{rem}

\section{An off-diagonal estimate for bianalytic Bergman kernels} \label{offdiagonalsection}

In this section we use the $\bar \partial^2$-estimate of 
\ref{rhoprop} to prove theorem \ref{offdiagonalthm}.
\begin{proof}[Proof of theorem \ref{offdiagonalthm}]
We fix  $z_0 \in \mathrm{int} \mathcal{S} \cap \mathcal{N}_+$ and use the same 
definitions for the parameters $r_0, \alpha, A, \beta, l, M_0, M_1$ as in section \ref{hormandersection}. 

We also take arbitrary $z_1 \in \mathcal{S}$, and require
\[ 9 m^{-1/2} \leq \text{min} \{ r_0, |z_1-z_0| \}. \]  
We then set
\[ r_1 := \frac12 \text{min} \{r_0, |z_1-z_0|-m^{-1/2} \} \]
and see that 
\begin{equation} \label{r1andm}
4m^{-1/2} \leq r_1 \quad \text{and} \quad \qquad 2r_1 \leq r_0.
\end{equation}
The two last conditions clearly imply the assumptions $2m^{-1/2}<r_1$ and $r_1<r_0$ that were
 made in the section \ref{hormandersection}. 
Furthermore, we will also require that $m \geq m_0$ 
and 
\begin{equation} \label{nandmcondition}
n \geq ]m_0-M_0+M_1+1[, 
\end{equation}
where $M_1 >2$ and $m_0$ is so big that  
the theorem \ref{rhoprop} holds. By making $M_0$ bigger if needed, 
we see that the assumptions $n \geq m-M$ and $m \geq m_0$ (which were made in the statement 
of the theorem) actually imply \eqref{nandmcondition}. 

We choose a smooth cut-off function $\chi$
that satisfies $\chi=0$ on $\mathbb{D}(z_0, r_1)$, $\chi=1$ on $\mathbb{C} \backslash 
\mathbb{D}(z_0, 2r_1)$ and $0 < \chi(w) < 1$ in the annulus $\mathbb{D}(z_0,2r_1) \backslash 
\mathbb{D}(z_0, r_1)$. We can arrange it so that 
\begin{equation} \label{dbarchi}
\int_{\mathbb{C}} |\bar \partial \chi(w)|^2 \diff A(w) \leq K_1
\end{equation}
and 
\begin{equation} \label{dbardbarchi}
\int_{\mathbb{C}} |\bar \partial^2 \chi(w)|^2 \diff A(w) \leq \frac{K_2}{r_1^2},
\end{equation}
for some absolute constants $K_1$ and $K_2$. 

We consider the equation
\[ 
\bar \partial^2 u = \bar \partial^2 [ K_{2,mQ,n}(\cdot, z_0) \chi(\cdot) ]. 
\] 
The $L^2_{2,mQ,n}$-minimal solution $u_{2,mQ,n}$ can be written as 
\[
u_{2,mQ,n}=K_{2,mQ,n}(w, z_0)\chi(w)-\mathrm{P}_{2,mQ,n}[K_{2,mQ,n}(\cdot, z_0) \chi(\cdot)](w).
\]
Because $\chi(z_0)=0$, we have 
\begin{multline}
\bigg[ \int_{\mathbb{C}} |K_{2,mQ,n}(w, z_0)|^2 \chi(w) e^{-mQ(w)} \diff A(w)
\bigg]^2 e^{-mQ(z_0)} \\
= |\mathrm{P}_{2,mQ,n}[K_{2,mQ,n}(\cdot, z_0) \chi(\cdot)](z_0)|^2e^{-mQ(z_0)} 
=|u_{2,mQ,n}(z_0)|^2e^{-mQ(z_0)},
\end{multline}
Because $\bar \partial^2 [ K_{2,mQ,n}(\cdot, z_0) \chi(\cdot) ] = 0$ in 
$\mathbb{D}(z_0, m^{-1/2})$, we can apply the inequality
\ref{pointwise1} to get
\begin{multline} 
|u_{2,mQ,n}(z_0)|^2e^{-mQ(z_0)} \\
\leq (8+48 A^2) e^A m \int_{\mathbb{D}(z_0, m^{-1/2})}
|u_{2,mQ,n}(w)|^2 e^{-mQ(w)} \diff A(w)  \\
\leq 
(8+48A^2)e^Am \int_{\mathbb{C}} |u_{2,mQ,n}(w)|^2 e^{\rho_m(w) - mQ(w)} \diff A(w).
\end{multline}
Here and later, $\rho_m$ is a any function which satisfies the hypotheses required by theorem 
\ref{rhoprop}. In particular, for the last inequality we use that $\rho_m = 0$ in $\mathbb{D}(z_0, m^{-1/2})$. 

We are now ready to apply theorem \ref{rhoprop}.  
\begin{multline}
(8+48A^2)e^A m \int_{\mathbb{C}} |u_{2,mQ,n}(w)|^2 e^{\rho_m(w) - mQ(w)} \diff A(w) 
 \\
\leq \frac{C_1}{m}
\int_{\mathbb{C}} |\bar \partial_w^2[K_{2,mQ,n}(w, z_0) \chi(w)]|^2e^{\rho_m(w)-mQ(w)} \diff A(w)\\
\leq \frac{8C_1}{m}
 \int_{r_1 \leq |w-z_0| \leq 2r_1} | \bar \partial_w K_{2,mQ,n}(w,z_0)|^2 |\bar \partial \chi(w)|^2 e^{\rho_m(w)-mQ(w)} \diff A(w) \\
 +  \frac{2C_1}{m}
 \int_{r_1 \leq |w-z_0| \leq 2r_1} |K_{2,mQ,n}(w,z_0)|^2 |\bar \partial^2 \chi(w)|^2 e^{\rho_m(w)-mQ(w)} \diff A(w),
\end{multline}
where the constant $C_1 > 0$ only depends on the parameters $A, \alpha, M_0, M_1, l$ and $\beta$. 
We analyze the two terms separately. For the first one, we estimate
\begin{multline}\label{offdiagonal-eq1}
\int_{r_1 \leq |w-z_0| \leq 2r_1} | \bar \partial_w K_{2,mQ,n}(w, z_0)|^2 \cdot |\bar \partial_w \chi(w)|^2 e^{\rho_m(w)- mQ(w)} \\
\leq \sup_{r_1\leq |w-z_0| \leq 2r_1} \bigg\{ | \bar \partial_w K_{2,mQ,n}(w, z_0)|^2 e^{\rho_m(w)- mQ(w)}  \bigg\} 
\int_{r_1 \leq |w-z_0| \leq 2r_1}  | \bar \partial_w \chi(w)|^2 \diff A(w) 
\end{multline}
From \eqref{pointwise2}, we get
\begin{multline}
 | \bar \partial_w K_{2,mQ,n}(w, z_0)|^2 e^{- mQ(w)} \\
\leq 3e^A m^2 \int_{|\tilde w - w| \leq m^{-1/2}} |K_{2,mQ,n}( \tilde w, z_0)|^2e^{-mQ(\tilde w)} \diff A( \tilde w) \\
\leq 3 e^A m^2 K_{2,mQ,n}(z_0, z_0).
\end{multline}
for any $w \in \mathbb{D}(z_0,2r_1)$. 
Furthermore, by \eqref{diagonalestimate1},  
\[
3 e^A m^2 K_{2,mQ,n}(z_0, z_0) \leq 3 e^{2A} (48+12A^2)m^3 e^{mQ(z_0)}.
\]
After applying \eqref{dbarchi} to the second factor of \eqref{offdiagonal-eq1}, we can conclude
\begin{multline}
\int_{r_1 \leq |w-z_0| \leq 2r_1} | \bar \partial_w K_{2,mQ,n}(w, z_0)|^2 \cdot |\bar \partial_w \chi(w)|^2 e^{\rho_m(w)- mQ(w)} \\
\leq 
C_2 m^3 \sup_{r_1\leq |w-z_0| \leq 2r_1} e^{\rho_m(w)} e^{mQ(z_0)}
\end{multline}
for a constant $C_2$ depending only on $\alpha, A, \beta$ and $M_0$. 

We now move to the second term of the last expression in \eqref{offdiagonal-eq1}. 
\begin{multline} \label{offdiagonalsecondterm}
\int_{r_1 \leq |w-z_0| \leq 2r_1} |K_{2,mQ,n}(w,z_0)|\bar \partial^2 \chi(w)|^2 e^{\rho_m(w)-mQ(w)} \diff A(w)
\\
\leq 
 \sup_{r_1\leq |w-z_0| \leq 2r_1} \bigg \{ e^{\rho_m(w)}   | K_{2,mQ,n}(w, z_0)|^2 e^{- mQ(w)}  \bigg\}
\int_{r_1 \leq |w-z_0| \leq 2r_1}  | \bar \partial^2 \chi(w)|^2 \diff A(w) 
\end{multline}
For the second factor, we use that  
\[ 
\int_{\mathbb{C}} |\bar \partial^2 \chi(w)|^2 \diff A(w) \leq \frac{K_2}{r_1^2}
\leq \frac{K_2\cdot m}{16},
\]
where we used \eqref{r1andm}.   
The first factor in \eqref{offdiagonalsecondterm} is analyzed with \eqref{pointwise1}
and \eqref{diagonalestimate1}:
\begin{multline}
 | K_{2,mQ,n}(w, z_0)|^2 e^{- mQ(w)} \\
\leq m(8+48A^2)e^A \int_{\mathbb{D}(z_0,m^{-1/2})} |K_{2,mQ,n}(\tilde w, z_0)|^2
e^{-mQ(\tilde w)} \diff A(w) \\
\leq m(8+48A^2)e^A K_{2,mQ,n}(z_0,z_0)
\leq m^2(8+48A^2)^2e^{2A}e^{mQ(z_0)}
\end{multline}
for any $w \in \mathbb{D}(z_0,2r_1)$. 

Summarizing the discussion so far, we have 
\begin{multline}
\bigg[ \int_{\mathbb{C}} |K_{2,mQ,n}(w, z_0)|^2 \chi(w) e^{-mQ(w)} \diff A(w)
\bigg]^2 e^{-mQ(z_0)} \\
\leq C_3m^2 \sup_{r_1 \leq |w-z_0| \leq 2r_1} \{ e^{\rho_m(w)} \} e^{mQ(z_0)}
\end{multline}
for a constant $C_3 > 0$ depending only on 
$\alpha, A,  \beta, l, M_0$ and $M_1$. This leads to 
\begin{equation}
\int_{\mathbb{C}}|K_{2,mQ,n}(w,z_0)|^2 \chi(w) e^{-mQ(w)} \diff A(w)e^{-mQ(z_0)}
\leq \sqrt{C_3} m \sup_{r_1 \leq |w-z_0| \leq 2r_1} \{ e^{\frac12 \rho_m} \}. 
\end{equation} 
We can use the estimate \eqref{pointwise1} 
to deduce 
\begin{equation} \label{offdiagonalestimate1}
|K_{2,mQ,n}(z_0,z_1)|^2e^{-mQ(z_0)-mQ(z_1)}
\leq C_4  m^2 \sup_{r_1 \leq |w-z_0| \leq 2r_1} \{ e^{\frac12 \rho_m} \},
\end{equation}
with $C_4$ depending only on 
$\alpha, A, \beta, l, M_0$ and $M_1$. To get a good 
off-diagonal estimate for $K_{2,mQ,n}$, 
it remains to choose a function $\rho_m$ such that the conditions of the section \ref{hormandersection} 
are satisfied and that $ \sup_{r_1 \leq |z-z_0| \leq 2r_1} \rho_m(z)$ is small. 
We present one possible way to choose $\rho_m$ 
in the end of this section. Here, we only need to know that this choice of  
$\rho_m(z)$ only depends on the distance $|z-z_0|$, and we have
\[
\sup_{r_1 \leq |z-z_0| \leq 2r_1} \rho_m(z) = \rho_m(w), \quad \rho_m(w) = -k (\sqrt{m}r_1-3)   
\]
whenever $|w-z_0|=r_1$. Here, $k$ is a constant which 
depends only on $\alpha, A, M_0, M_1, \beta$ and $l$.  
By a simple computation, we get
\[
\rho_m(w) \leq -\frac12 k \sqrt{m} \min \{ r_0, |z_1-z_0| \} + \frac72 k.
\]
We have now obtained the estimate
\[
|K_{2,mQ,n}(z_0,z_1)|^2e^{-mQ(z_0)-mQ(z_1)}
\leq C e^{-\frac14 k \sqrt m \text{min} \{r_0, |z_1-z_0| \} }.
\]
for a constant $C$ depending only on the parameters 
$\alpha, A, M_0, M_1, \beta$ and $l$.
Recall that this was proven under the assumption
\begin{equation} \label{mcond}
 9m^{-1/2} \leq \text{min} \{ r_0, |z_1-z_0| \}. 
\end{equation}
We now investigate what happens when this condition does 
not hold.
Let us recall the estimate \eqref{apriorioffdiagonal}:
\[
|K_{2,mQ,n}(z_0,z_1)|^2e^{-mQ(z)-mQ(w)} \leq m^2 (8+48A^2)^2e^{2A}.
\]
This and the negation of \eqref{mcond} imply that 
\[
|K_{2,mQ,n}(z_0,z_1)|^2e^{-mQ(z_0)-mQ(z_1)}
\leq C m^2e^{-\sqrt m \text{min} \{r_0, |z_1-z_0| \} }
\]
holds with 
\[
C = (8+48A^2)^2e^{2A+9}.
\]

We have now showed the conclusion of the theorem \ref{offdiagonalthm} for a fixed 
$z_0$ in the interior of $\mathcal{S} \cap \mathcal{N}_+$. The constants $C$ and $\epsilon$ 
depend only on $M_0, M_1, \alpha, A, \beta$ and $l$. The number $m_0$ 
may additionally depend on $r_0$. 

Finally, it is clear from the proof that we can find $C$ and $\epsilon$ which work uniformly over any 
$z_0$ in a given compact set $\mathrm{K}$ 
in the interior of $\mathcal{S} \cap \mathcal{N}_+$. In fact, $r_0$ and $\alpha$ are the only parameters that 
depend on $z_0$, and they can be replaced by 
\[
r_{0,\mathrm{K}} := \frac14 \dist(\mathrm{K}, \mathbb{C} \backslash \mathcal{S} \cap \mathcal{N}_+).
\]
and
\[
\alpha_{\mathrm{K}} := \inf_{w \in \widehat{\mathrm{K}}} \Delta Q(w).
\] 
where 
\[
\widehat{\mathrm{K}}:= \overline{\cup_{z_0 \in \mathrm{K}} \mathbb{D}(z_0, 2r_0)}
\]
The theorem is now proved. 
\end{proof}
We now give one possible construction for the function $\rho_m$ that was used in the theorem. This will be done 
under the assumption $4m^{-1/2} \leq r_1$, which was also made in the proof. 

We set 
\[ \rho_m(z) = -k \sigma_m(|z-z_0|), \]
where $k$ is a constant to be specified later, and
\begin{align}
\sigma_m(x) = \begin{cases} 
0, & 0<x \leq m^{-1/2} \\
\frac12 m(x-m^{-1/2})^2, & m^{-1/2} < x \leq 2m^{-1/2} \\
\sqrt{m}(x-2m^{-1/2})+ \frac12, & 2m^{-1/2} < x \leq r_1-2m^{-1/2} \\
-\frac12m(x-r_1+m^{-1/2})^2+\sqrt{m}r_1-3,&  r_1-2m^{-1/2} < x \leq r_1-m^{-1/2} \\
\sqrt{m}r_1-3,& r_1-m^{-1/2} < x. 
\end{cases}
\end{align}
One can quickly verify that this function is $C^{1,1}$-smooth and the derivatives satisfy 
$\sigma_m'(x) \leq \sqrt m$ for all $x >0$ and $|\sigma_m''(x)| \leq m$ for all almost every $x > 0$. Now, we compute the 
derivatives of $\rho_m$ which
are relevant from the point of view of the conditions \eqref{rhocond1}, \eqref{rhocond2} and \eqref{rhocond3}: 
\begin{gather}
\bar \partial \rho_m(z)= 
-k\frac{(z-z_0)}{2|z-z_0|}\sigma_m'(|z-z_0|), \\
\bar \partial^2 \rho_m(z) =
k\bigg[ \frac{(z-z_0)^2}{4|z-z_0|^3} \sigma_m'(|z-z_0) 
- \frac{(z-z_0)^2}{4|z-z_0|^2} \sigma_m''(|z-z_0) \bigg], \\
\Delta \rho_m(z) = -\frac{k}{4} \bigg[ \frac{1}{|z-z_0|}\sigma_m'(|z-z_0|) + \sigma_m''(|z-z_0|) \bigg].
\end{gather}
We can assume $|z-z_0|\geq m^{-1/2}$, because $\rho_m=0$ otherwise. It is then easily checked that absolute value of the first of these derivatives is uniformly bounded by $k\sqrt{m}/2$ for all $z$. The two others are bounded by $km/2$ almost everywhere. 
Now, 
allowing $k$ to depend on the parameters 
$\alpha, A, M_0, M_1, \beta$ and $l$, 
we can fix it to be so small that  
the conditions \eqref{rhocond1}, \eqref{rhocond2} and \eqref{rhocond3} will be satisfied. 

\section{Remarks on the case $q>2$} \label{generalqsection}
In this section, we aim to explain how one could prove theorem \ref{blowupthm} for more general $q>2$. The main issue here 
is to identify those terms in the local $q$-analytic Bergman kernel $\mod(m^{-\frac12})$ which 
are the most dominant after the blow-up 
\[ z=z_0 + \frac{\xi}{\sqrt{m \Delta Q(z_0)}},  \quad
w=z_0 + \frac{\lambda}{\sqrt{m \Delta Q(z_0)}}
\] 
around $z_0 \in \mathcal{S} \cap \mathcal{N}_+$. 

We take $z_0$, $r$ and $\chi$ as in section \ref{localkernelsection}. 
We also recall the kernel $R_{q,m}$ from proposition \ref{reproprop}:
\begin{multline} 
R_{q,m}(z,w)  = m \bigg[\sum_{k=0}^{q-1} 
\frac{q! (-1)^k}{(q-1-k)! (k+1)! k!}
(\bar z - \bar w)^k \bar \partial_w^k \big( \bar \partial_w \theta e^{m(z-w) \theta} \big) \bigg]
e^{-m(z-w) \theta}.  \\
\label{Rqmformula}
\end{multline}

Let us define functions $A_{k,j}(z,w)$ by 
\[
\bar \partial_w^k  \big( \bar \partial_w \theta e^{m(z-w) \theta} \big) = 
\bigg[\sum_{j=0}^k m^j(z-w)^j A_{k,j}(z,w) \bigg]e^{m(z-w) \theta}. 
\]
The functions $A_{k,j}$ are analytic in $z$ and $\bar w$,  
and $A_{k,k} = (\bar \partial_w \theta)^{k+1}$. We can now rewrite \eqref{Rqmformula} as
\begin{equation} R_{q,m}(z,w)  
= m \sum_{k=0}^{q-1} \sum_{j=0}^k
\frac{q! (-1)^k}{(q-1-k)! (k+1)! k!}
(\bar z - \bar w)^k m^j (z-w)^j A_{k,j}(z,w)   
\label{Rqmformula3}
\end{equation}

We say that a term $A(z,w)$ is of order $l- \frac12 k - \frac12 j$, if we can write 
\begin{equation} \label{factorization1}
A(z,w)= m^{l} (\bar z - \bar w)^k (z-w)^j X_{l, k, j}(z,w), 
\end{equation}
where $X_{l,k,j}(z,w)$ is analytic in $z$, real-analytic in $w$ and cannot be divided by 
$\bar z - \bar w$ or $z-w$ so that the divided function would also satisfy these two properties. 
We also use the notation
\[
\mathrm{ord}[A] = l-\frac12 k - \frac12 j. 
\]
We define the order of a finite sum of such terms to be the maximum of the orders 
of the individual terms. 

The formula \eqref{Rqmformula3} shows that all terms in $R_{q,m}$ are of order $\leq 1$, 
and the order $1$ terms are those for which $j=k$. The order $1$ terms can be written 
more compactly as 
\[ m \bar \partial_w \theta L^1_{q-1} \big( m \bar \partial_w \theta |z-w|^2 \big), \]  
where $L^1_{q-1}$ is the associated Laguerre polynomial 
with parameter $1$ and degree $q-1$. The remaining terms are of order $\leq \frac12$. 

We will need a generalization of the proposition \ref{partintprop}. 
\begin{prop} \label{partintpropgeneralq}
For any $C^q$-smooth function $A: \mathbb{D}(z_0, r) \times \mathbb{D}(z_0, r) \to \mathbb{C}$, 
there exists $\delta > 0$ such that for any   
$u \in A^2_{q,mQ}$ and $z \in \mathbb{D}(z_0, \frac13 r)$ we have 
\begin{multline} \label{partintlemmageneralq}
\int u(w) \chi(w) m^q (z-w)^q A(z,w) e^{m(z-w) \theta} \diff A(w) \\
= \int u(w) \chi(w) \bigg[ \sum_{j=0}^{q-1} m^j(z-w)^j \sum_{k=0}^{q-j} X_{k,j}(z,w) \bar \partial_w^k A(z,w) \bigg]
\times  e^{m(z-w) \theta} \diff A(w) \\
+ \mathrm{O}( \|u \|_m e^{\frac12 m Q(z) - \delta m})
\end{multline}
for certain functions $X_{k,j}(z,w)$ which are analytic in $z$ and $\bar w$ and can be written 
as a rational functions of $\bar \partial_w$-derivatives of $\bar \partial_w \theta$. The number 
$\delta$ can depend on $Q$, $z_0$, $r$ and $A$. 
\end{prop}
\begin{proof}
Set $j=0$. We compute using lemma \ref{neglilemma}
\begin{multline}
\int \bar \partial_w^j u(w) \chi(w) m^{q-j} (z-w)^{q-j} \frac{A(z,w)}{(\bar \partial_w \theta)^j} e^{m(z-w) \theta} \diff A(w) \\
= \int \bar \partial_w^j u(w) \chi(w)  m^{q-j-1} (z-w)^{q-j-1} \frac{A(z,w)}{(\bar \partial_w \theta)^{j+1}} \bar \partial_w\big( e^{m(z-w)\theta} \big) \diff A(w) \\
= -\int \bar \partial_w^{j+1} u(w) \chi(w)  m^{q-j-1} (z-w)^{q-j-1} \frac{A(z,w)}{(\bar \partial_w \theta)^{j+1}}  e^{m(z-w)\theta} \diff A(w) \\
- \int \bar \partial_w^j u(w) \chi(w)  m^{q-j-1} (z-w)^{q-j-1} \bar \partial_w \bigg[\frac{A(z,w)}{(\bar \partial_w \theta)^{j+1}}\bigg]  e^{m(z-w)\theta} \diff A(w) \\
+ \mathrm{O}( \|u \|_m e^{\frac12 m Q(z) - \delta m}) \\
= -\int \bar \partial_w^{j+1} u(w) \chi(w)  m^{q-j-1} (z-w)^{q-j-1} \frac{A(z,w)}{(\bar \partial_w \theta)^{j+1}}  e^{m(z-w)\theta} \diff A(w) \\
+(-1)^{j+1} \int  u(w) \chi(w)  m^{q-j-1} (z-w)^{q-j-1} \bar \partial_w^j \bigg[ \bar \partial_w \bigg(\frac{A(z,w)}{(\bar \partial_w \theta)^{j+1}}\bigg)  e^{m(z-w)\theta}\bigg] \diff A(w) \\
+ \mathrm{O}( \|u \|_m e^{\frac12 m Q(z) - \delta m}) \\
\end{multline}
The second integral will then be left as such and the first integral will be analyzed again in the same way after setting $j \to j+1$. The process stops when $j=q$. 
\end{proof}
The algorithm to compute local $q$-analytic Bergman kernels involves Taylor expanding $R_{q,m}$ in powers
of $(z-w)$ and 
using this proposition. The proposition will only be applied to functions $A(z,w)$ which 
can be written in the form \eqref{factorization1}. 
We then have 
\[ \mathrm{ord} \big[ \bar \partial_w^k A(z,w) \big] \leq \mathrm{ord} A + \frac{k}{2}, \]  
and therefore
\[ 
\mathrm{ord} \bigg[ \sum_{j=0}^{q-1} m^j (z-w)^j \sum_{k=0}^{q-j} X_{k,j}(z,w) \bar \partial_w^k A(z,w) \bigg] 
\leq \mathrm{ord} \bigg[ m^q (z-w)^q A \bigg].
\]
In other words, applying the proposition \ref{partintpropgeneralq}
to a term of the form $m^q (z-w)^q A(z,w)$ cannot increase the order. 

Suppose that we are after a local $q$-analytic Bergman kernel $\mod(m^{-\frac12})$, which will be in the form
\begin{multline}
K^{\langle q \rangle}_{q,m}(z,w) = \Lfun^{\langle q \rangle}_{q,m}(z,w)e^{mQ(z,w)} \\
= \big(m^q \Lfun_{q,0}(z,w)+m^{q-1}\Lfun_{q,1}(z,w)+\dots+ m\Lfun_{q,q-1}(z,w) \big)e^{mQ(z,w)}.
\end{multline}
Proceeding as in the case $q=2$, the first step is to Taylor expand $R_{q,m}(z,w)$ in powers of $(z-w)$ with 
\eqref{dbarthetataylor}. One gets the presentation
\begin{align} \label{Rqmdecompo}
R_{q,m}(z,w) &= \sum_{l=1}^q \sum_{k=1}^{q-1} \sum_{j=1}^{q+l-1} m^l (\bar z- \bar w)^k (z-w)^j D_{l,k,j}(z,w) \\
&+\sum_{l=1}^q  \sum_{k=1}^{q-1} m^l (\bar z- \bar w)^k (z-w)^{q+l} E_{l,k}(z,w),
\end{align}
where the functions $D_{j,k,l}$ are $q$-analytic in $(z,\bar w)$ and $E_{l.k}$ are 
$q$-analytic in $z$ and real-analytic in $w$.  
Applying the previous proposition $l$ times shows that the terms containing functions 
$E_{l.k}$ contribute $\mathrm{O}\big(m^{-1/2}\|u \|_m e^{\frac12 mQ(z)}\big)$ and can therefore 
be neglected for the purpose at hand.
 
Taylor expansion of the order $1$ terms of $R_{q,m}$ yields 
\begin{multline} 
m \bar \partial_w \theta L^1_{q-1} \big( m \bar \partial_w \theta |z-w|^2 \big) 
= m b(z,w) L^1_{q-1} \big( m b(z,w) |z-w|^2 \big) \\
+ \text{ terms of order 1/2 or less}+ \text{negligible terms}. 
\end{multline}
On the other hand, the terms which remain from the Taylor expansion of order $\leq \frac12$ terms 
of $R_{q,m}$ also are of order $\leq \frac12$. We can therefore write
\begin{multline} \label{Rqmtaylorexpansion1}
R_{q,m}(z,w) = m b(z,w) L^1_{q-1} \big( m b(z,w) |z-w|^2) \\
+ \text{ terms of order 1/2 or less}+ \text{negligible terms}. 
\end{multline}
The order $1$ terms on the right hand side are $q$ analytic in $z$ and $\bar w$, so they will be left intact in 
the subsequent steps of the algorithm. The process continues by applying prosition \ref{partintpropgeneralq} 
to those terms in \eqref{Rqmdecompo} that contain $D_{l,k,j}$ for $j \geq q$ (the other ones are either
negligible or $q$-analytic in $\bar w$) and then Taylor expanding the result in powers of $(z-w)$ again. 
Then one ends up with a presentation of the form \eqref{Rqmdecompo} again, but now 
the highest possible power for the parameter $m$ is $q-1$. 

A crucial point is that further steps of the algorithm  
cannot produce any new order $1$ terms. This follows from two facts. The first one is that 
if one Taylor expands a term $m^l(\bar z - \bar w)^k (z-w)^j X(z,w)$, where $X$ is 
some function which is $q$-analytic in $z$ and real-analytic $w$, in powers of $(z-w)$, 
then each of the resulting terms has order which is less or equal to the order of 
$m^l(\bar z - \bar w)^k (z-w)^j X(z,w)$.
The second fact is that the application of proposition \ref{partintpropgeneralq} cannot increase the order. 
Therefore, when the algorithm is finished (i.e. when all terms that are not $q$-analytic in $\bar w$
are of the form $m^{l}(\bar z- \bar w)^k(z-w)^{j}X(z,w)$ for $l \leq 0$),
we have produced a local poly-Bergman kernel 
$K^{\langle q \rangle}_{q,m}(z,w)= \Lfun^{\langle q \rangle}_{q,m}(z,w)e^{mQ(z,w)}$, so that 
\begin{equation}
 \Lfun^{(q)}_{q,m}(z,w) = m b(z,w) L^1_{q-1} \big( m b(z,w) |z-w|^2 \big) 
+ \text{ terms of order 1/2 or less}.
\end{equation}

Let us now put 
\[ z=z_0 + \frac{\xi}{\sqrt{m \Delta Q(z_0)}}, \quad 
w=z_0 + \frac{\lambda}{\sqrt{m \Delta Q(z_0)}}
\] 
and apply the Taylor expansions 
\[ 
b(z,w) = \Delta Q(z_0) + \mathrm{O}(m^{-1/2})
\]
and 
\begin{equation} \label{blowupexponential}
e^{m\mathrm{Re}Q(z,w)-\frac12m\big(Q(z)+Q(w)\big)}
= e^{-\frac12 |\xi-\lambda|^2} + \mathrm{O}\big(m^{-\frac12}\big)
\end{equation}
which hold uniformly for $(\xi, \lambda)$ in any given compact subset of $\mathbb{C}^2$ 
when $m$ is so big that $z,w \in \mathbb{D}(z_0,\frac13r)$. We get 
\begin{equation}
\frac{1}{m \Delta Q(z_0)} \bigg| K^{(q)}_{q,m}(z,w) e^{mQ(z,w)}\bigg| 
= \big|L^1_{q-1} \big(|\xi - \lambda|^2 \big)\big| e^{-\frac12 |\xi-\lambda|^2} +\mathrm{O}(m^{-\frac12}),
\end{equation} 
where the result holds uniformly for all $(\xi, \lambda)$ in any given compact subset of 
$\mathbb{C}^2$. 

To show that the same result holds when the local polyanalytic Bergman kernel $K^{(q)}_{q,m}$ is replaced by $K_{q,mQ,n}$, 
one should proceed as in section \ref{blowupsection}. 
One would also need to generalize the proposition \ref{rhoprop} to solutions of $\bar \partial^q$-equations 
for general $q$ (the special case $\rho_m=0$ would suffice here). We expect this to be a straightforward but 
rather laborious task. 

The reader may have noticed that apart from the construction of local polyanalytic Bergman kernels, 
the lemma \ref{pointwiselemma} has been a crucial ingredient in our arguments. 
The proof method presented in \cite{hh2} should work also for more general $q$-analytic functions, 
but argument would become rather cumbersome. 
Therefore, it might be of value to quickly sketch another way of proving similar statements 
which easily works for any $q$.
We fix a point $z$ and consider the Riesz decomposition
\[
Q(w)= h(w)+ \mathbf{G}[\Delta Q](w),  \quad w \in \mathbb{D}(z, m^{-1/2}),
\]
where $h$ is a harmonic function and $\mathbf{G}[\Delta Q](w)$ is the Green's potential of $\Delta Q$ in the disk
$\mathbb{D}(z, m^{-1/2})$. 
One can then find an analytic function $H_m$ such 
that $|H_m|^2= e^{-mh}$
, and we get   
\[
|H_m(w)|^2 \leq C e^{-mQ(w)} \qquad |z-w| \leq m^{-1/2},
\] 
for some constant $C$. Using this fact, one can easily make a reduction to the case $Q=0$. The desired inequalities in 
this unweighted case could be worked out with the work of Koshelev \cite{koselev}, who 
idenfied explicit reproducing kernels for polyanalytic Bergman spaces 
on the unit disk with constant weight. 

\section{Appendix: the third term of the bianalytic Bergman kernel}
We will here compute the third term of the local bianalytic Bergman kernel. The Taylor expansion \ref{dbarthetataylor}, lemma \ref{neglilemma} and proposition \ref{partintprop} 
will be casually used without further notice. We start as before by 
Taylor expanding \eqref{bianalrepro}. 
\begin{multline} \label{thirdterm1}
u(z) 
= \int_{\mathbb{C}} u(w) \chi(w)  \bigg[ 2m \bigg( b + \frac12 (w-z) \partial_z b) -m (\bar z - \bar w) 
(\partial_w b + \frac12 (w-z)\bar \partial_w \partial_z b \bigg) \\
- m^2|z-w|^2b^2 
+ m^2(z-w)^2 A_1(z,w) + m(z-w)^2 B_1(z,w) \\
+ m^2(z-w)^3 C_1(z,w)  \bigg]
e^{m(z-w) \theta} \diff A(w) + \mathrm{O} \big( \| u \|_m e^{\frac12 m Q(z)} m^{-\frac32} \big),
\end{multline}
where 
\begin{multline}
A_1(z,w) = (\bar z - \bar w) b \partial_z b \quad B_1(z,w) =  \frac13 \partial_z^2 b - \frac16 (\bar z - \bar w) \bar \partial_w \partial_z^2 b, \\ 
\quad C_1(z,w) = (\bar w - \bar z) \bigg( \frac13 b \partial_z^2 b + \frac14 (\partial_z b)^2 \bigg).
\end{multline}
Only these last three terms require further analysis.  We analyze $B_1$ first:
\begin{multline}
m \int_{\mathbb{C}} u(w) \chi(w) (z-w)^2 B_1(z,w) e^{m(z-w) \theta} \diff A(w) \\
= -\int_{\mathbb{C}} u(w) \chi(w)  \bigg[ \frac1m \bar \partial_w^2 \frac{B_1}{(\bar \partial_w \theta)^2}
+ (z-w)\bigg(  2 \frac{\bar \partial_w B_1}{\bar \partial_w \theta}-3 \frac{B_1 \bar \partial_w^2 \theta }{(\bar \partial_w \theta)^2}
\bigg) \bigg] e^{m(z-w)\theta} \diff A(w) \\
+ \mathrm{O} \big( \| u \|_m e^{\frac12 m Q(z)} e^{-\delta m} \big) \\
= \int_{\mathbb{C}}u(w) \chi(w) (z-w) \bigg[
-\frac23 \frac{\bar \partial_w \partial_z b}{b}
- \frac13 \frac{\bar \partial_w \partial_z^2 b}{b}
+ \frac{\partial_z^2 b \cdot \bar \partial_w b}{b^2} \\
+ (\bar z - \bar w) \bigg( \frac13 \frac{\bar \partial_w^2 \partial_z^2 b}{b}
+ \frac12 \frac{\bar \partial_w \partial_z^2 b \cdot \bar \partial_w b}{b^2} \bigg) \bigg] e^{m(z-w) \theta} \diff A(w)
+ \mathrm{O} \big( \| u \|_{m} e^{\frac12 m Q(z)} m^{-\frac32} \big).
\end{multline}
We turn to the term with $C_1$. 
\begin{multline} \label{C2}
\int_{\mathbb{C}} u(w) \chi(w) m^2(z-w)^3 C_1 e^{m(z-w) \theta} \diff A(w) \\
= -\int_{\mathbb{C}} u(w) \chi(w) \bigg[ (z-w) \bar \partial_w^2 \frac{C_1}{(\bar \partial_w \theta)^2} + m(z-w)^2 \bigg( 2 \frac{\bar \partial C_1}{\bar \partial_w \theta} - 3 \frac{C_1 \bar \partial_w^2 \theta}{(\bar \partial_w \theta)^2} \bigg) \bigg]
e^{m(z-w)\theta} \diff A(w) \\
+ \mathrm{O} \big( \| u \|_m e^{\frac12 m Q(z)} e^{-\delta m} \big) \\
= -\int_{\mathbb{C}} u(w) \chi(w) \bigg[ (z-w) \bar \partial_w^2 \frac{C_1}{b^2} 
+ m(z-w)^2 C_2(z,w) \bigg]
e^{m(z-w)\theta} \diff A(w),  \\
\end{multline}
where
\begin{multline}
-\bar \partial_w^2 \frac{C_1}{b^2}
= -\bar \partial_w^2 \bigg[ (\bar w - \bar z) \bigg( 
\frac13 \frac{\partial_z^2 b}{b} + \frac14 \frac{(\partial_z b)^2}{b^2}
\bigg) \bigg] \\
= - \frac23 \frac{\bar \partial_w \partial_z^2 b}{b}
- \frac{\partial_z b \cdot \bar \partial_w \partial_z b}{b^2}
+ \frac{(\partial_z b)^2 \cdot \bar \partial_w b}{b^2} 
+\frac23 \frac{\partial_z^2 b \cdot \bar \partial_w b}{b^2} \\
+(\bar z - \bar w) \bigg[ 
\frac13 \frac{\bar \partial_w^2 \partial_z^2 b}{b}
-\frac23 \frac{\bar \partial_w \partial_z^2 b \cdot \bar \partial_w b}
{b^2} 
-\frac13 \frac{\partial_z^2 b \cdot \bar \partial_w^2 b}{b^2}\\
+\frac23 \frac{\partial_z^2 b \cdot (\bar \partial_w b)^2}{b^3}
+ \frac12 \frac{(\bar \partial_w \partial_z b)^2}{b^2}
+ \frac12 \frac{\partial_z b \cdot \bar \partial_w^2 \partial_z b}{b^2}
- \frac{\partial_z b \cdot \bar \partial_w \partial_z b \cdot 
\bar \partial_w b}{b^3}\\
- \frac{\partial_z b \cdot \bar \partial_w \partial_z b \cdot \bar \partial_w b}{b^3} - \frac12 \frac{(\partial_z b)^2 \bar \partial_w^2 b}{b^3}
+ \frac32 \frac{(\partial_z b) (\bar \partial_z b)^2}{b^4} \bigg],
\end{multline}
and 
\begin{multline}
C_2 :=  2 \frac{\bar \partial_w C_1}{\bar \partial_w \theta} 
- 3 \frac{C_1 \bar \partial_w^2 \theta}{(\bar \partial_w \theta)^2} \\
= \frac23 \partial_z^2b + \frac12 \frac{(\partial_z b)^2}{b} 
+ (\bar w- \bar z) \bigg[ \frac 23 \frac{\bar \partial_w b\cdot \partial_z^2 b}{b} 
+ \frac23 \cdot \bar \partial_w \partial_z^2 b
+ \frac{ \partial_z b \cdot \bar \partial_w \partial_z b }{b} \bigg] \\
-(\bar w - \bar z) \bigg( \frac{\partial_z^2 b}{b} 
+ \frac34 \frac{(\partial_z b)^2}{b^2} \bigg) \bar \partial_w b.
\end{multline} 
We continue by analyzing \eqref{C2} by with proposition \ref{partintprop}. 
\begin{multline}
-\int_{\mathbb{C}} u(w) \chi(w) m(z-w)^2 C_2 e^{m(z-w)\theta} \diff A(w) \\
= 
\int_{\mathbb{C}} u(w) \chi(w) \bigg[ \frac1m \bar \partial_w^2 \frac{C_2}{(\bar \partial_w \theta)^2} + 
(z-w) \bigg( 2 \frac{\bar \partial_w C_2}{\bar \partial_w \theta} 
- 3 \frac{C_2 \bar \partial_w^2 \theta}{(\bar \partial_w \theta)^2} \bigg) \bigg] e^{m(z-w) \theta} \diff A(w) \\
+ \mathrm{O} \big( \| u \|_{m} e^{\frac12 m Q(z)} e^{-\delta m} \big) \\
= 
\int_{\mathbb{C}} u(w) \chi(w) \bigg[ \frac1m \bar \partial_w^2 \frac{C_2}{(\bar \partial_w \theta)^2} + 
(z-w) \bigg( 2 \frac{\bar \partial_w C_2}{b} 
- 3 \frac{C_2 \bar \partial_w b}{b^2} \bigg) \bigg] e^{m(z-w) \theta} \diff A(w) \\
+ \mathrm{O} \big( \| u \|_{m} e^{\frac12 m Q(z)} m^{-\frac32} \big).
\end{multline}
The first term is negligible for our purposes. The second term equals 
\begin{multline}
2 \frac{ \bar \partial_w C_2}{b} - 3 \frac{ C_2 \bar \partial_w b}{b^2} \\
=  \frac83 \frac{\bar \partial_w \partial_z^2 b}{b}
+ 4 \frac{ \partial_z b \cdot \bar \partial_w \partial_z b}{b^2}
-\frac83 \frac{\bar \partial_w b \cdot \partial_z^2 b}{b^2}
-4 \frac{ (\partial_z b)^2 \cdot \bar \partial_w b}{b^3} \\
+ (\bar w - \bar z) \bigg[
-\frac23 \frac{\bar \partial_w^2 b \cdot \partial_z^2 b}{b^2}
-\frac83 \frac{\bar \partial_w b \cdot \bar \partial_w \partial_z^2 b}{b^2}
+\frac53 \frac{(\bar \partial_w b)^2 \cdot \partial_z^2 b}{b^3} +
\frac43 \frac{\bar \partial_w^2 \partial_z^2 b}{b} 
+2 \frac{ (\bar \partial_w \partial_z b)^2}{b^2} \\
+2 \frac{\partial_z b \cdot \bar \partial_w^2 \partial_z b}{b^2} 
-8 \frac{\partial_z b \cdot \bar \partial_w \partial_z b \cdot \bar \partial_w b}{b^3}
-\frac32 \frac{(\partial_z b)^2 \bar \partial_w^2 b}{b^3} 
+ \frac{21}{4} \frac{(\partial_z b)^2 (\bar \partial_w b)^2}{b^4} \bigg].
\end{multline}
We now turn to the term $A_1$ in \eqref{thirdterm1}. 
\begin{multline}
\int_{\mathbb{C}} u(w) \chi(w) m^2(z-w)^2 A_1 \diff A \\
= \int_{\mathbb{C}} u(w) \chi(w) \bigg[ - \partial_w^2 \frac{A_1}{(\bar \partial_w \theta)^2 } 
+ m(z-w) \bigg( -2 \frac{ \bar \partial_w A_1}{\bar \partial_w \theta} + 3 \frac{A_1 \bar \partial_w^2 \theta}{(\bar \partial_w \theta)^2} \bigg) \bigg] e^{m(z-w) \theta} \\ 
+ \mathrm{O} \big( \| u \|_{m} e^{\frac12 m Q(z)} e^{-\delta m} \big) \\
= \int_{\mathbb{C}} u(w) \chi(w) \bigg\{ A_2 + m(z-w) \bigg[ - 2\frac{1}{b} 
\bar \partial_w \bigg( (\bar z - \bar w) \partial_z b \cdot b \bigg) 
+ 3 (\bar z - \bar w) \frac{\partial_z b \cdot \bar \partial_w b}{b} \bigg] \\
+ m(z-w)^2 A_3 \bigg\} e^{m(z-w) \theta} \diff A(w) 
+ \mathrm{O} \big( \| u \|_{m} e^{\frac12 m Q(z)} m^{-\frac32} \big), \\ 
\end{multline}
where
\begin{multline} \label{A2}
A_2 :=  -\bar \partial_w^2 \bigg[ (\bar z - \bar w) \partial_z b \cdot b 
\bigg( \frac{1}{b^2} - (w-z) \frac{ \partial_z b}{b^3} \bigg) \bigg] \\
= 2 \frac{\bar \partial_w \partial_z b}{b} 
-2\frac{\partial_z b \cdot \bar \partial_w b}{b^2} 
\\
+(\bar z - \bar w)
\bigg[ - \frac{\bar \partial_w^2 \partial_z b}{b}
+ 2 \frac{\bar \partial_w \partial_z b \cdot \bar \partial_w b}{b^2}
+ \frac{\partial_z b \cdot \bar \partial_w^2 b}{b^2}
-2 \frac{\partial_z b \cdot (\bar \partial_w b)^2}{b^3} \bigg] \\
+(w-z) \bigg\{ -2 \frac{\partial_z b \cdot \bar \partial_w \partial_z b}{b^2}
+2 \frac{(\partial_z b)^2 \cdot \bar \partial_z b}{b^3}
\\
+(\bar z - \bar w) 
\bigg[ 2 \frac{(\bar \partial_w \partial_z b)^2}{b^2}
+2 \frac{\partial_z b \cdot \bar \partial_w^2 \partial_z b}{b^2}
-8 \frac{\partial_z b \cdot \bar \partial_w \partial_z b \cdot 
\bar \partial_w b}{b^3}  \\
-2 \frac{(\partial_z b)^2 \cdot \bar \partial_w^2 b}{b^3}
+6 \frac{(\partial_z b)^2 \cdot (\bar \partial_w b)^2}{b^4} 
\bigg] \bigg\}
\end{multline}
and 
\begin{multline}
A_3 := -\frac{\partial_z b}{b^2} \bar \partial_w \bigg( (\bar z - \bar w) \partial_z b \cdot b \bigg) 
+3(\bar z - \bar w) \frac{ (\partial_z b)^2 \bar \partial_w b}{b^2} + \frac32 (\bar z - \bar w) \frac{ \partial_z b \cdot \bar \partial_w \partial_zb}{b} \\
=\frac{(\partial_z b)^2}{b} + (\bar z - \bar w)
\bigg[ -\frac52 \frac{\partial_z b \cdot \bar \partial_w \partial_z b}{b}
+2 \frac{(\partial_z b)^2 \bar \partial_w b}{b^2} \bigg].
\end{multline}
We continue with the term $A_3$. 
\begin{multline}
\int_{\mathbb{C}} u(w) \chi(w) m(z-w)^2 A_3 e^{m(z-w) \theta} \diff A(w) \\
= \int_{\mathbb{C}} u(w) \chi(w) \bigg\{- \frac1m \bar \partial_w^2 
\frac{A_3}{\bar \partial_w \theta} + (z-w) 
\bigg[ -2 \frac{\bar \partial_w A_3}{(\bar \partial_w \theta)^2} 
+ 3 \frac{A_3 \bar \partial_w^2 \theta}{\bar \partial_w \theta}\bigg] \bigg\} \\
e^{m(z-w)} \diff A(w) + \mathrm{O} \big( \| u \|_{m} e^{\frac12 m Q(z)} e^{-\delta m} \big) \\
= \int_{\mathbb{C}} u(w) \chi(w) \bigg\{- \frac1m \bar \partial_w^2 
\frac{A_3}{b^2} + (z-w) 
\bigg[ -2 \frac{\bar \partial_w A_3}{b} 
+ 3 \frac{A_3 \bar \partial_w b}{b}\bigg] \bigg\} 
e^{m(z-w)} \diff A(w) \\ 
+ \mathrm{O} \big( \| u \|_{m} e^{\frac12 m Q(z)} m^{-\frac32} \big).
\end{multline} 
The first term is negligible, and the second term equals 
\begin{multline} \label{A4}
 -2 \frac{\bar \partial_w A_3}{b^2} 
+ 3 \frac{A_3 \bar \partial_w b}{b}
= -9 \frac{\partial_z b \cdot \bar \partial_w \partial_z b}{b^2}
+9 \frac{(\partial_z b)^2 \bar \partial_w b}{b^3} \\
+( \bar z - \bar w) \bigg[ -5\frac{(\bar \partial_w \partial_z b)^2}{b^2}
-5 \frac{\partial_z b \cdot \bar \partial_w^2 \partial_z b}{b^2}
-\frac{41}{2} \frac{\partial_z \cdot \bar \partial_w \partial_z b \cdot \bar \partial_w b}{b^3} 
\\
- 4 \frac{(\partial_z b)^2 \bar \partial_w^2 b}{b^3}
+14 \frac{(\partial_z b)^2 \cdot (\bar \partial_w b)^2}{b^4} \bigg].
\end{multline} 
This concludes the analysis of $A_1$ from \eqref{thirdterm1}. We combine everything together within \eqref{thirdterm1} 
to get the third 
term (i.e. constant order contribution) for the bianalytic Bergman kernel. We omit the laborious summation.  
\[\Lfun_{2,2} = 2 \bar \partial_w \partial_z \log b + 
(\bar w - \bar z) \bar \partial_w^2 \partial_z \log b
+ (z-w) \bar \partial_w \partial_z^2 \log b
+|z-w|^2 M(z,w), \]
where 
\begin{multline}
M = 
+\frac32 \frac{\bar \partial_w^2 \partial_z b \cdot \partial_z b}{b^2} 
-\frac{13}{2} \frac{\partial_z b \cdot \bar \partial_w \partial_z b \cdot \bar \partial_w b}{b^3}
+\frac32 \frac{(\bar \partial_w \partial_z b)^2}{b^2}\\
-\frac{(\partial_z b)^2 (\bar \partial_w^2 b)}{b^3}
+ \frac{17}{4} \frac{(\partial_z b)^2 \cdot (\bar \partial_w b)^2}{b^4}
-\frac23 \frac{\partial_z^2 \bar \partial_w^2 b}{b}
+\frac32 \frac{\partial_z^2 \bar \partial_w b \cdot \bar \partial_w b}{b^2} \\
-\frac{(\partial_z^2 b)(\bar \partial_w b)^2}{b^3}
+\frac13 \frac{\bar \partial_w^2 b \cdot \partial_z^2 b}{b^2}.
\end{multline}
\subsection*{Acknowledgements}
The author would like to thank H\aa kan Hedenmalm for helpful suggestions, and Bo Berndtsson for a
useful discussion at CRM in Barcelona.


\begin{thebibliography}{1}
\bibitem{abreu} L. Abreu, \emph{Sampling and interpolation in Bargmann-Fock spaces of polyanalytic functions}, 
Appl. Comp. Harm. Anal. {\bf 29} (2010), 287--302.

\bibitem{agz}
G. W. Anderson, A. Guionnet, O. Zeitouni,  \textit{An introduction to random 
matrices}. Cambridge Studies in Advanced Mathematics, {\bf118}. Cambridge 
University Press, Cambridge, 2010. 

\bibitem{ahm1} 
Y. Ameur, H. Hedenmalm, N. Makarov, \emph{Berezin transform in polynomial Bergman spaces},
Comm. Pure Appl. Math. {\bf 63} (2010), no. 12, 1533--1584.

\bibitem{ahm2} 
 Y. Ameur, H. Hedenmalm, N. Makarov, \emph{Fluctuations of eigenvalues of random normal
matrices}, Duke Mathematical Journal {\bf 159} (2011), 31--81.

\bibitem{ahm3} Y. Ameur, H. Hedenmalm, N. Makarov, \emph{Random normal matrices 
and Ward identities}, preprint at  arXiv:1109.5941.

\bibitem{ameurortega} Y. Ameur, J. Ortega-Cerd\`{a}, 
\emph{ Beurling-Landau densities of weighted Fekete sets and correlation kernel estimates}, J. Funct. Anal. {\bf 263}  (2012), no. 7, 1825--1861.

\bibitem{berman}
R.Berman, 
\textit{Bergman kernels for weighted polynomials and weighted equilibrium measures of $\mathbb{C}^n$}, Indiana University Mathematics Journal {\bf 58} (2009), issue 4.

\bibitem{berman2}
R. Berman, 
\textit{Determinantal point processes and fermions on complex manifolds: Bulk universality}, preprint at arXiv:0811.3341.

\bibitem{bbs} R. Berman, B. Berndtsson, J. Sj\"ostrand, 
\emph{A direct approach to Bergman kernel asymptotics for positive line bundles}, Arkiv f\"or matematik {\bf 46} (2008), no. 2, 197--217. 

\bibitem{bor} A. Borodin, \textit{Determinantal Point Processes}. 
arXiv:0911.1153v1. 

\bibitem{cat} D. Catlin, \textit{The Bergman kernel and a theorem of Tian}, Analysis and geometry in several complex variables, 
(Katata, 1997), 1999, 1--23.

\bibitem{deift} P. Deift, \textit{Orthogonal polynomials and random 
matrices: a Riemann-Hilbert approach}. Courant Lecture Notes in Mathematics, 
{\bf3}. New York University, Courant Institute of Mathematical Sciences, 
New York; American Mathematical Society, Providence, RI, 1999. 

\bibitem{deiftgioev} P. Deift, D. Gioev, \textit{Random matrix theory: 
invariant ensembles and universality}. Courant Lecture Notes in Mathematics, 
{\bf18}. Courant Institute of Mathematical Sciences, New York; American 
Mathematical Society, Providence, RI, 2009. 

\bibitem{deift2} Deift, P., \textit{Universality for mathematical and physical systems}, 
Congress of Mathematicians. Vol. I, 125--152, Eur. Math. Soc., Z\"urich, 2007.

\bibitem{dunne} G.V. Dunne, \textit{
International Journal of Modern Physics B}, vol. 8 (1994), 1625--1638.

\bibitem{englis} M. Engli\^s, \emph{The asymptotics of a Laplace integral on a K\"ahler manifold}, 
J. Reine Angew. Math. {\bf 528} (2000), 1--39.

\bibitem{Forr} P.J. Forrester,  \textit{Log-gases and random matrices}.
London Mathematical Society Monograph Series, {\bf 34}. Princeton University
Press, Princeton, NJ, 2010.

\bibitem{HedMak1} H. Hedenmalm, N. Makarov,  \textit{Quantum Hele-Shaw flow}.
arXiv: math/0411437.

\bibitem{HedMak2} H. Hedenmalm, N. Makarov, 
\textit{Coulomb gas ensembles and Laplacian growth}. Proc. London Math. Soc., 
to appear.

\bibitem{HedShi} H. Hedenmalm, S. Shimorin, \textit{Hele-Shaw flow on 
hyperbolic surfaces}. J. Math. Pures Appl. (9) {\bf81} (2002), no. 3, 
187--222.

\bibitem{hh} A. Haimi, H. Hedenmalm, \textit{The Polyanalytic Ginibre Ensembles}, preprint at arXiv:1106.2975.

\bibitem{hh2} A. Haimi, H. Hedenmalm , \textit{Asymptotic expansion of Polyanalytic Bergman kernels}, preprint.

\bibitem{hkpv} J. B. Hough, M.  Krishnapur, Y. Peres, B. Vir\'{a}g,  
\textit{Determinantal processes and independence}. Probab. Surv. {\bf3} 
(2006), 206--229. 

\bibitem{joh} K. Johansson, \textit{On fluctuations of eigenvalues of random Hermitian 
matrices}, Duke Math. J. {\bf91} (1998), no.1, 151--204. 

\bibitem{koselev} 
A. D. Koshelev, \textit{The kernel function of a Hilbert space of functions that are polyanalytic
in the disc}, Dokl. Akad. Nauk SSSR {\bf232} (1977), no. 2, 277--279

\bibitem{ks} 
D. Kinderlehrer, G. Stampacchia, \textit{An introduction to variational inequalities and their applications. Pure and Applied
Mathematics}, 88. Academic Press, 1980.

\bibitem{lindholm}
N. Lindholm, \emph{Sampling in weighted $L^p$ spaces of entire functions in $\mathbb{C}^N$ and estimates 
of the Bergman kernel}, J. Funct. Anal {\bf 182} (2001), 390--426. 

%


\bibitem{macchi} O. Macchi, \textit{The coincidence approach to stochastic 
point processes}. Advances in Appl. Probability {\bf7} (1975), 83--122.

\bibitem{mehta} M. L. Mehta, \textit{Random matrices}.
Third edition. Pure and Applied Mathematics (Amsterdam), {\bf142}. 
Elsevier/Academic Press, Amsterdam, 2004. 

\bibitem{mo} Z. Mouayn, \emph{New formulae representing magnetic Berezin transforms as functions of the
Laplacian on $\mathbb{C}^N$}, preprint at arXiv:1101.0379.

\bibitem{mm} X. Ma, G. Marinescu, \textit{
Holomorphic Morse inequalities and Bergman kernels}. Birkh\"auser Verlag AG, Basel, 2007.

\bibitem{ortegaseip} J. Ortega-Cerd\`{a}, K. Seip, \emph{Beurling-type density theorems for weighted $L^p$ spaces of entire functions},
J. Analyse Mathematique {\bf 75} (1998), 247--266.

\bibitem{st} E. B. Saff, V. Totik, \textit{Logarithmic potentials with external fields}. 
Grundlehren der mathematischen Wissenschaften, {\bf 316}. Springer, Berlin, 1997

\bibitem{tian} G. Tian, \emph{On a set of polarized K\"ahler metrics on algebraic manifolds}, J. Diff. Geom. {\bf 32} (1990), 99--130.

\bibitem{vasil} N. L. Vasilevski, \textit{Poly-Fock spaces} Oper. Theory Adv. Appl. {\bf 117} 
 (2000), 371--386.

\bibitem{yau} S.-T. Yau, \emph{Nonlinear analysis in geometry}, L'Enseignement Math. {\bf 33} (1987), 109--158


\bibitem{zabro1} A. Zabrodin,  \textit{Matrix models and growth processes: 
from viscous flows to the quantum Hall effect}. 
Applications of random matrices in physics, 261--318, NATO Sci. Ser. II 
Math. Phys. Chem., {\bf221}, Springer, Dordrecht, 2006.
arXiv:hep-th/0412219v1.

\bibitem{zel} S. Zelditch, \emph{Szeg\"o kernel and a theorem of Tian}, Internat. Math. Res. Notices{\bf 6} (1998), 317--331.


\end{thebibliography}
\end{document}